\documentclass[10pt]{amsart}
\usepackage{fancyhdr}
\usepackage{etoolbox}
\IfFileExists{stmaryrd.sty}{\usepackage{stmaryrd}}{%
}
\usepackage{calrsfs}
\usepackage{amsmath,amssymb,amsthm,mathtools}
\usepackage{listings}
\usepackage{xcolor}
\usepackage[nospace,noadjust]{cite}
\usepackage{amsfonts}
\usepackage[framemethod=default]{mdframed}
\usepackage{enumerate}
\usepackage{comment}
\usepackage[all]{xy}
\usepackage{tcolorbox}
\tcbuselibrary{theorems}
\usepackage{lineno}
\usepackage{bm}
\usepackage{graphicx}
\usepackage{psfrag}
\usepackage{url}
\usepackage{tikz}
\usepackage{tikz-cd}
\usepackage{dirtytalk}
\usepackage{float}
\usepackage{mathrsfs}
\usepackage{subfiles}
\usepackage{caption}
\usepackage{subcaption}
\usepackage{array}
\usepackage{microtype}
\usepackage{placeins}
\definecolor{wine}{RGB}{114,47,55}
\usepackage{hyperref}
\usepackage{cleveref}
\newtheoremstyle{spacedstyle}
 {\medskipamount}
 {\medskipamount}
 {\itshape}
 {}
 {\bfseries}
 {.}
 {.5em}
 {}

\newtheoremstyle{spacedefinition}
 {\medskipamount}
 {\medskipamount}
 {\normalfont}
 {}
 {\bfseries}
 {.}
 {.5em}
 {}

\theoremstyle{spacedstyle}
\newtheorem{theorem}{Theorem}[section]

\newtheorem{proposition}[theorem]{Proposition}
\newtheorem{question}[theorem]{Question}

\newtheorem{conjecture}[theorem]{Conjecture}
\newtheorem{lemma}[theorem]{Lemma}

\theoremstyle{spacedefinition}
\newtheorem{definition}[theorem]{Definition}
\newtheorem{remark}[theorem]{Remark}
\newtheorem{example}[theorem]{Example}

\newcommand{\T}{\mathcal T}
\newcommand{\stir}[2]{\genfrac\{\}{0pt}{}{#1}{#2}}
\DeclareMathOperator{\stnd}{st}

\hypersetup{
 pdftitle={Contraction Containment in Labeled Trees: Support Counts, Collision Cores, and Star Avoidance},
 pdfauthor={Levi Segal},
 colorlinks=true,
 linkcolor=blue,
 citecolor=cyan,
 urlcolor=wine
}

\title[Contraction containment in labeled trees]
{Contraction Containment in Labeled Trees:\\
Support Counts, Collision Cores, and Star Avoidance}
\author{Levi Segal}
\address{Stanford Online High School, 415 Broadway Academy Hall, Floor 2, 8853, 415 Broadway, Redwood City, CA 94063, USA}
\email{levisegal0@gmail.com}
\date{}

\begin{document}

\begin{abstract}
We study whether a labeled tree can be recovered from the numbers of larger
labeled trees that contract to it.  
We first characterize containment using the edge splits induced on a
proposed survivor set and derive an exact formula for hosts with one marked
display. We then study several displays in the same host.
After contracting every edge invisible to all \(k\) displays of an
\(n\)-vertex target, the remaining collision core has at most
\(k(n-1)+1\) vertices.  Once the coincidences and relative positions of the
surviving labels are recorded, these bounded cores give a finite exact
expansion of each collision count.
We also prove that almost every sufficiently large labeled tree
contains any fixed target, with an exponentially small exceptional
proportion. We obtain an exact avoidance formula and generating function for star
displays in which the center survivor is the largest label of the
host. We call such displays top-centered. We derive fixed-star asymptotics
and a coarse threshold for this restricted containment problem.
\end{abstract}
\maketitle
\tableofcontents

\section{Introduction and Main Results}\label{sec:introm}
Ordinary edge contraction merges the endpoints of an edge without specifying
which label survives.  For labeled trees, that
choice is part of the combinatorial data.  We therefore direct a contraction
toward the endpoint whose label is retained.  After a sequence of directed
contractions, the surviving labels are replaced by
\(1,2,\ldots\) in increasing order. We call this \emph{standardization}.  If a tree \(U\) on \([m]\) can be
reduced in this way to a tree \(T\) on \([n]\), we write
\[
T\preceq U
\]
and say that \(U\) contains \(T\).  The tree \(T\) is the \emph{target} and
\(U\) is the \emph{host}. The retained vertices, with their original labels,
are called \emph{survivor vertices}.

A contraction leading from \(U\) to \(T\) partitions the vertices of \(U\)
into \(n\) connected blocks.  Each block is contracted to one retained
vertex, and the increasing order of these \(n\) survivors determines the
labels of the quotient.  Informally, a \emph{display} records these connected
blocks and the survivor chosen in each one. The precise definition is given
in Section~\ref{sec:displays}.

Let \(\T_m\) denote the set of labeled trees on \([m]\).  For a fixed target
\(T\in\T_n\), its \emph{support count} at order \(m\) is
\[
\mu_T(m)=\#\{U\in\T_m:T\preceq U\}.
\]
Thus every host contributes once, independently of the number of displays it
admits.  In this paper, we are interested in the extent to which the sequence
\(\bigl(\mu_T(m)\bigr)_{m\ge n}\) determines the target.

Directed labeled contraction has a reversal symmetry. Let
\(T^{\mathrm{rev}}\) be obtained from $T$ by relabeling each vertex
\(i\mapsto n+1-i\). Reversing the host labels by $j\mapsto m+1-j$
carries a display of $T$ to a display of $T^{\mathrm{rev}}$. This gives a
bijection between their containing hosts at every order, so
\[
\mu_T(m)=\mu_{T^{\mathrm{rev}}}(m)
\qquad(m\ge n).
\]
This motivates the following conjecture.

\begin{conjecture}[Reversal-Wilf rigidity]
\label{conj:RW}
For $T_1,T_2\in\T_n$,
\[
\mu_{T_1}(m)=\mu_{T_2}(m)\quad\text{for every }m\ge n
\quad\Longrightarrow\quad
T_2=T_1\ \text{or}\ T_2=T_1^{\mathrm{rev}}.
\]
\end{conjecture}

We call two targets with the same support sequence \emph{Wilf equivalent},
by analogy with permutation patterns; see Simion and
Schmidt~\cite{SimionSchmidt} and B{\'o}na~\cite{Bona}.  Dotsenko
\cite{Dotsenko} studied pattern avoidance in planar rooted labeled trees
under a different containment relation.  Our objects are unrooted Cayley
trees, and containment is defined by directed contraction.
Conjecture~\ref{conj:RW} remains open. We explain how the results below relate
to it at the end of this section.

\subsection{Multiplicity and displays}

The first obstruction to computing \(\mu_T(m)\) is multiplicity.  Let
\(c_T(U)\) be the number of displays of \(T\) in \(U\), and define the
factorial collision moments
\[
C_k(T;m)=\sum_{U\in\T_m}\binom{c_T(U)}{k}.
\]
Thus \(C_1(T;m)\) counts a host together with one marked display, whereas
\(\mu_T(m)\) counts only the underlying hosts.  Hostwise
inclusion--exclusion gives
\[
\mu_T(m)=\sum_{k\ge1}(-1)^{k+1}C_k(T;m),
\]
where the sum is finite for each \(m\).

For instance, take \(T\) to be the path \(1-2-3\) and \(U\) to be the path
\[
1-4-2-5-3.
\]
Contracting \(4\to1\) and \(5\to3\) gives one display.  Contracting
\(4\to2\) and \(5\to2\) gives another. These are two distinct displays,
while the host contributes only once to the support count.

\subsection*{Survivor splits and block expansion}

Deleting an edge of \(T\) divides \([n]\) into the label sets of the two
resulting components.  These bipartitions form the edge-split system
\(\Sigma(T)\).  Now fix a proposed survivor set
\[
R=\{r_1<\cdots<r_n\}\subseteq[m]
\]
in a host \(U\). Every edge whose deletion leaves survivors in both
components separates \(R\) into two nonempty classes.  Replacing \(r_i\) by its
index \(i\) gives a split of \([n]\); the set of all such splits is denoted by
\(\Sigma_R(U)\). Section~\ref{sec:displays} defines these split systems
formally and proves the following criterion.

\begin{theorem}\label{thm:intro-split}
Let \(T\in\T_n\) and \(U\in\T_m\).  Then \(T\preceq U\) if and only if
there is an \(n\)-element survivor set \(R\subseteq[m]\) such that
\[
\Sigma(T)\subseteq\Sigma_R(U).
\]
\end{theorem}

For a fixed display, let \(b_i\) be the size of the contraction block
representing target vertex \(i\).  The labels outside the survivor set are
distributed among these blocks; Cayley's formula counts each internal block
tree; and every target edge selects one attachment vertex in each of its two
incident blocks.  This yields the following exact marked count.

\begin{theorem}\label{thm:intro-marked}
Let \(T\in\T_n\), and write \(d_T(i)\) for the degree of \(i\) in \(T\).
For every \(m\ge n\),
\[
C_1(T;m)=\binom{m}{n}
\sum_{\substack{b_1+\cdots+b_n=m\\b_i\ge1}}
\binom{m-n}{b_1-1,\ldots,b_n-1}
\prod_{i=1}^n b_i^{\,b_i-2+d_T(i)}.
\]
\end{theorem}

Theorems~\ref{thm:intro-split} and \ref{thm:intro-marked} are proved as
Theorems~\ref{thm:split_system} and \ref{thm:marked_display}.  The second
formula determines \(C_1(T;m)\), but not \(\mu_T(m)\), because a host may
admit multiple displays.

\subsection*{Bounded collision cores}

To analyze collisions, we consider several displays in the same host
simultaneously.  Although the host may have arbitrarily many vertices, only
a bounded part of it is visible to a fixed number of display maps.  Indeed,
suppose that $\delta_1,\ldots,\delta_k$ are distinct displays of an
$n$-vertex target in a common host, with display maps $\pi_1,\ldots,\pi_k$.
If the endpoints of a host edge have the same image under every $\pi_a$,
then all $k$ display maps factor through its contraction.
Contract all such edges.  Every edge that remains crosses a contraction-block
boundary for at least one display.  Since each display has exactly $n-1$
crossing edges between its $n$ connected blocks, the reduced tree has at
most $k(n-1)$ edges, and hence at most
$k(n-1)+1$ vertices.  Thus the simultaneous quotient structure of any fixed
number of displays is encoded by a tree of bounded order, independently of
the size of the host. We recover hosts by replacing the vertices of a core
with connected labeled trees and choosing attachment vertices. We call
these constructions \emph{expansions} of the core.

The next theorem gives the bound and the resulting collision formula.
To count the displays exactly, we must also record the survivor choices.
Sections~\ref{sec:cores} and~\ref{sec:decorated-cores} define these
configurations and their expansion counts.

\begin{theorem}
\label{thm:intro-cores}
Fix a target tree $T\in\mathcal T_n$ and an integer $k\geq 1$.  Let
$\delta_1,\ldots,\delta_k$ be pairwise distinct displays of $T$ in a common
host $U$.  Contract every host edge whose endpoints lie in the same
contraction block in each of the $k$ displays.  The resulting quotient tree
has at most
\[
k(n-1)+1
\]
vertices.

For fixed $T$ and $k$, only finitely many quotient configurations can arise
in this way once the locations, coincidences, and required order relations of
the survivor labels are recorded.  Moreover, for every $m$, the number
\[
k!\,C_k(T;m)
\]
of hosts on $[m]$ equipped with an ordered $k$-tuple of distinct displays is
given by a finite sum of explicit expansion counts, one for each such
configuration.
\end{theorem}
The exact form of the finite sum is as follows.  Let
$\mathcal C_k(T)$ denote the finite set of isomorphism classes of the bounded
configurations described in the theorem, including the data of the survivor
labels.  For a representative $\Omega$, let
$L_m^{\mathrm{rig}}(\Omega)$ be the number of expansions to the label set
$[m]$ after the bounded quotient has been identified with that representative,
and let $\operatorname{Aut}(\Omega)$ be its group of symmetries preserving all
display and survivor data.  Then
\[
k!\,C_k(T;m)
=
\sum_{[\Omega]\in\mathcal C_k(T)}
\frac{L_m^{\mathrm{rig}}(\Omega)}
     {|\operatorname{Aut}(\Omega)|}.
\]
The numerator counts expansions with a chosen identification of their bounded
quotient with $\Omega$.  Dividing by
$|\operatorname{Aut}(\Omega)|$ removes the multiplicity introduced by that
choice.  The configurations $\Omega$, their automorphisms, and the lift counts
$L_m^{\mathrm{rig}}(\Omega)$ are defined formally in
Section~\ref{sec:decorated-cores}.

The number of possible configurations depends only on $T$ and $k$, not on
the host order $m$.  The dependence on
$m$ enters only when the vertices of a bounded configuration are expanded
into connected labeled subtrees and the survivor labels and attachment points
are chosen.  Theorem~\ref{thm:decorated-lifts} gives the resulting exact
formula.

This finiteness holds at each fixed collision order \(k\). It does not
bound the number of collision orders needed to recover \(\mu_T(m)\)
independently of \(m\). Inclusion--exclusion may involve every collision
order permitted by a host on \([m]\).

\subsection*{Two consequences}

The first consequence concerns large hosts.  A uniform Cayley tree has a
linear expected number of exact standardized copies of a fixed target.
Moon's forest-extension formula shows that copies on disjoint vertex sets
are independent in the sense needed for a sparse dependency graph.  Suen's
inequality then makes simultaneous avoidance exponentially unlikely.

\begin{theorem}
For every fixed labeled tree $T\in\T_n$, there exists $c_T>0$ such that
\[
0\le 1-\frac{\mu_T(m)}{m^{m-2}}=O_T(e^{-c_Tm}).
\]
\end{theorem}

Thus every fixed target has containment density tending to one exponentially
fast.  Any rigidity in Conjecture~\ref{conj:RW} must therefore be carried by
the avoidance counts or by finer information about display counts, not
by the limiting density.

The second consequence treats stars. Let $S_n$ be the star on $[n+1]$
whose center is $n+1$. In every display of $S_n$, the center survivor is
the largest among the $n+1$ survivors because it represents the target
vertex $n+1$. It need not, however, be the largest label of the host. We
call a display \emph{top-centered} when its center survivor is the absolute
largest host label $k$. This is a restriction of ordinary containment,
where any suitable host label may survive as the center. For top-centered
displays, containment is controlled by the number of host leaves, with one
boundary case when that number is exactly $n$.

\begin{theorem}
Let $n\ge3$, and let $S_n$ be the star on $[n+1]$ with center $n+1$.
For $k\ge2$, let $A_n^{\mathrm{top}}(k)$ be the number of
trees $U\in\mathcal T_k$ that contain no top-centered display of $S_n$. If $\tau(k,L)$ denotes the number of labeled
trees on $[k]$ with exactly $L$ leaves, then
\[
A_n^{\mathrm{top}}(k)=\sum_{L=0}^{n-1}\tau(k,L)+\frac{n}{k}\tau(k,n).
\]
Equivalently, with
\[
T_L(x)=\sum_{k\ge2}\tau(k,L)\frac{x^k}{k!},
\qquad
A_n^{\mathrm{top}}(x)=\sum_{k\ge2}A_n^{\mathrm{top}}(k)\frac{x^k}{k!},
\]
one has
\[
A_n^{\mathrm{top}}(x)=\sum_{L=0}^{n-1}T_L(x)
+n\int_0^x\frac{T_n(t)}{t}\,dt.
\]
Moreover $A_n^{\mathrm{top}}(x)$ is rational up to one logarithmic term, and
for fixed $n\ge3$,
\[
A_n^{\mathrm{top}}(k)\sim
\frac{k!}{2^{n-2}(n-2)!(n-1)!}\,k^{2n-5}.
\]
If $U_k$ is a uniformly random tree on $[k]$ and $n=n(k)$ satisfies
$3\le n(k)<k$ for all sufficiently large $k$, then the event
that $U_k$ contains a top-centered display of $S_{n(k)}$ has threshold
$n\sim k/e$. Its probability tends to $1$ when
\[
    \limsup_{k\to\infty}\frac{n(k)}{k}<\frac1e
\]
and to $0$ when
\[
    \liminf_{k\to\infty}\frac{n(k)}{k}>\frac1e.
\]
\end{theorem}

\subsection*{Relation to Conjecture~\ref{conj:RW}}

Conjecture~\ref{conj:RW} asks how much of a labeled target can be recovered
from the number of containing hosts at each order. The split criterion
replaces contraction sequences with a condition on the host. Counting one
marked display then gives an explicit formula, but that formula depends
only on the degree multiset of the target. It therefore cannot distinguish
all targets. Any further information in the support sequence must come from
the way different displays overlap in the same host. The bounded-core
theorem reduces these overlaps, at each fixed collision order, to finitely
many configurations. The collision formula lets us study the information
that the one-display count loses.

The asymptotic results give a second limitation. Every fixed target has
containment density tending exponentially quickly to one, so the limiting
density cannot distinguish targets. The relevant information must lie in
the avoidance counts, or equivalently in the differences between the
support counts. The top-centered star problem gives a case in which these
avoidance counts can be determined exactly. Requiring the largest host
label to survive as the center reduces containment to a condition on the
number of leaves. This gives formulas for fixed stars and a threshold for
growing stars. These results do not settle Conjecture~\ref{conj:RW} or
ordinary star avoidance, where any suitable label may survive as the center.

Sections~\ref{sec:displays} and~\ref{sec:marked}
develop the display formulation, the split criterion, and the one-display
enumeration.  Sections~\ref{sec:cores} and~\ref{sec:decorated-cores} analyze
overlapping displays through bounded configurations and exact collision
counts.  Section~\ref{sec:asymptotic} establishes exponential containment
for every fixed target.  The remaining sections study top-centered star displays, always requiring
the center survivor to be the absolute largest host label. They begin with
exact avoidance and continue through generating functions, fixed-star
asymptotics, and a growing-star threshold.
\section{Displays and split systems}\label{sec:displays}
For an ordered set \(R=\{r_1<\cdots<r_n\}\), write
\[
\stnd_R:R\longrightarrow[n],
\qquad
\stnd_R(r_i)=i,
\]
for its increasing \emph{standardization}.
\begin{definition}\label{def:directed-contraction}
Let \(T=(V,E)\) be a tree on $[n]=\{1,2,\ldots,n\}$, and let
$ab\in E(T)$. A \emph{directed contraction} \(a\to b\) deletes \(a\)
and the edge \(ab\), replaces every other edge $ax$ by $bx$, and retains
the label \(b\). The remaining labels are then standardized to $[n-1]$.
For a sequence of contractions, standardizing after each step or after
the full sequence gives the same result, since the relative order of the
surviving labels is preserved. In describing contraction sequences, we
retain the original labels until the final standardization.
\end{definition}

\begin{definition}\label{def:display}
Let \(T\in\T_n\) and \(U\in\T_m\).  A \emph{display} of \(T\) in \(U\) is a
pair \(\delta=(R,q)\) consisting of an \(n\)-element set
\[
R=\{r_1<\cdots<r_n\}\subseteq[m]
\]
and a retraction \(q:[m]\to R\), so $q(r_i)=r_i$ for each $i$, with
the following properties:
\begin{enumerate}
\item every fiber \(q^{-1}(r_i)\) induces a connected subtree of \(U\);
\item after each fiber is contracted to one vertex, the quotient tree,
      relabeled by \(r_i\mapsto i\), is \(T\).
\end{enumerate}
The elements of \(R\) are the \emph{survivors}, \(q\) is the
\emph{destination map}, and its fibers are the \emph{contraction blocks}, sometimes referred to as connected blocks or quotient blocks.
The associated \emph{display map} is the surjection
\[
\pi=\stnd_R\circ q:[m]\twoheadrightarrow[n].
\]
We write \(T\preceq U\) when \(U\) admits at least one display of \(T\).
The vertices of $T$ are called \emph{target vertices}.
\end{definition}

The destination map remembers actual host labels, whereas the display map
remembers only their standardized labels.  A connected block does not
by itself determine its survivor.

\begin{proposition}
\label{prop:display-equivalence}
A pair \((R,q)\) is a display of \(T\) in \(U\) if and only if there is a
sequence of directed contractions that deletes \([m]\setminus R\), sends each
vertex \(v\) to \(q(v)\), and leaves a tree whose increasing standardization
is \(T\).  Consequently, two contraction sequences determine the same
display precisely when they have the same survivor set and destination map.
\end{proposition}

\begin{proof}
Every vertex merged into a survivor is joined to it through contracted
edges, so a contraction sequence produces the contraction blocks.  Contracting
those blocks gives its final quotient. Thus every contraction sequence
determines a display.

Conversely, root the subtree induced by \(q^{-1}(r_i)\) at \(r_i\).  Repeatedly
contract a leaf of this rooted subtree toward its parent, working from the
leaves toward the root.  Doing this independently in every fiber deletes
exactly \([m]\setminus R\) and sends each vertex to its prescribed
destination.  The quotient condition in Definition~\ref{def:display} then
gives \(T\) after standardization.
\end{proof}

\begin{proposition}\label{prop:partial-order}
Directed-contraction containment is a partial order on
\(\bigsqcup_{n\ge1}\T_n\).
\end{proposition}

\begin{proof}
Reflexivity follows by contracting no edges. If $T\preceq U$ and
$S\preceq T$, contract $U$ to $T$ and then $T$ to $S$. Thus $S\preceq U$.
Finally, \(T\preceq U\) implies \(|V(T)|\le |V(U)|\).  If both
\(T\preceq U\) and \(U\preceq T\), the orders are equal, so neither display
contracts an edge and \(T=U\).
\end{proof}

The following example shows two displays with the same survivor set.

\begin{example}\label{ex:pathdisplay}
Let $T$ be the labeled path $1-2-3$, and let $U$ be the labeled path
\[
1-4-2-5-3.
\]
Choose the survivor set \(R=\{1,2,3\}\).  The destination map with fibers
\(\{1,4\}\), \(\{2\}\), and \(\{3,5\}\) gives one display; it can be
realized by the contractions \(4\to1\) and \(5\to3\).  The same host has a
second display with fibers \(\{1\}\), \(\{2,4,5\}\), and \(\{3\}\),
realized by $4\to2$ and $5\to2$.
\end{example}

\begin{figure}[H]
\centering
\begin{tikzpicture}[scale=.95, every node/.style={circle,draw,inner sep=1.3pt,minimum size=18pt}]
\node (a) at (0,0) {$1$};
\node (b) at (1.2,0) {$4$};
\node (c) at (2.4,0) {$2$};
\node (d) at (3.6,0) {$5$};
\node (e) at (4.8,0) {$3$};
\draw (a)--(b)--(c)--(d)--(e);
\draw[dashed,rounded corners] (-.35,-.45) rectangle (1.55,.45);
\draw[dashed,rounded corners] (2.05,-.45) rectangle (2.75,.45);
\draw[dashed,rounded corners] (3.25,-.45) rectangle (5.15,.45);
\node[draw=none,rectangle] at (2.4,-1.05) {$U$ displays $1-2-3$ by two directed contractions.};
\end{tikzpicture}
\caption{A contraction display.  Dashed boxes indicate the connected contraction blocks determined by the destination map.}
\label{fig:display}
\end{figure}

\begin{definition}\label{def:split-systems}
For an edge \(e\) of \(T\), deleting \(e\) produces two components with
label sets \(A_e\) and \(A_e^c=[n]\setminus A_e\). The unordered bipartition
\(A_e\mid A_e^c\) is the \emph{edge split} of \(e\).  The
\emph{edge-split system} of \(T\) is
\[
\Sigma(T)=\{A_e\mid A_e^c:e\in E(T)\}.
\]

Now let \(U\in\T_m\) and let
\(R=\{r_1<\cdots<r_n\}\subseteq[m]\).  If deleting an edge \(g\) of \(U\)
leaves survivors on both sides, let \(I_g\subset[n]\) be the indices of the
survivors in $R$ in one component.  The resulting unordered proper split is
\(I_g\mid I_g^c\).  The \emph{survivor split system} is
\[
\Sigma_R(U)
=
\{I_g\mid I_g^c:g\in E(U),\ I_g\ne\varnothing,\ I_g^c\ne\varnothing\}.
\]
It is a set, so repeated splits induced by different host edges are recorded
once.  Singleton-versus-complement splits are retained; only edges with all
survivors on one side are omitted.
\end{definition}

\begin{lemma}\label{lem:edge_splits}
The edge split system $\Sigma(T)$ determines $T$.
\end{lemma}

\begin{proof}
For distinct labels $i,j\in[n]$, set
\[
d_\Sigma(i,j)=\#\{A\mid A^c\in\Sigma(T): |\{i,j\}\cap A|=1\}.
\]
An edge separates $i$ from $j$ exactly when it lies on the unique $i$--$j$ path in $T$.  In particular, $ij\in E(T)$ if and only if $d_\Sigma(i,j)=1$. Thus the entire edge set of $T$ is determined.
\end{proof}

\begin{theorem}\label{thm:split_system}
Let $T\in\T_n$ and $U\in\T_m$.  Then $T\preceq U$ if and only if there is an ordered survivor set $R=\{r_1<\cdots<r_n\}$ such that
\[
\Sigma(T)\subseteq\Sigma_R(U).
\]
\end{theorem}

\begin{proof}
If $\pi:U\twoheadrightarrow T$ is a display map with survivor set $R$,
then every edge of $T$ comes from a unique edge of $U$ joining two
contraction blocks. Thus every split of $T$ is present in $\Sigma_R(U)$.

Conversely, suppose \(\Sigma(T)\subseteq\Sigma_R(U)\), and let \(S_R(U)\)
be the minimal subtree of \(U\) containing \(R\).  Every host edge inducing a
proper survivor split belongs to \(S_R(U)\).  For each
\(\sigma\in\Sigma(T)\), choose an edge \(g_\sigma\) in $S_R(U)$ inducing \(\sigma\) in $\Sigma_R(U)$.
Distinct edges of \(T\) have distinct edge splits.  Indeed, if \(e\ne f\),
then \(f\) lies entirely in one component of \(T-e\); because \(e\) remains
after \(f\) is deleted, the components of \(T-f\) cannot equal those of
\(T-e\).  A host edge induces only one split, so the chosen edges
\(g_\sigma\) are also distinct.

There are \(n-1\) chosen edges.  Deleting them from \(U\) therefore produces
exactly \(n\) connected components.  No component contains two survivors.
Indeed, if \(r_i\) and \(r_j\) remained in the same component, no chosen edge
would separate them.  On the other hand, the target path from \(i\) to \(j\)
contains an edge whose split separates \(i\) and \(j\), and its chosen host
edge would separate \(r_i\) and \(r_j\).  Since the \(n\) survivors lie in
\(n\) components, each component contains exactly one survivor.

Contract each component toward its unique survivor.  The quotient has one
edge for each \(g_\sigma\), and that edge induces the split \(\sigma\) after
standardization.  Its edge-split system is therefore \(\Sigma(T)\).
Lemma~\ref{lem:edge_splits} identifies the standardized quotient with \(T\),
so \((R,q)\) is a display.
\end{proof}

\section{Counting one marked display}\label{sec:marked}

Before counting containing trees, we count pairs $(U,\delta)$ in which
$\delta$ is one chosen display of $T$ in $U$.  This is the part of the
problem for which the contraction blocks can be counted separately.

\begin{definition}\label{def:collision-moments}
For $T\in\T_n$ and $U\in\T_m$, the \emph{display multiplicity} is
\[
c_T(U)=\#\{\text{displays of }T\text{ in }U\}.
\]
For $k\ge1$, the $k$th \emph{factorial collision moment} is
\[
C_k(T;m)=\sum_{U\in\T_m}\binom{c_T(U)}{k}.
\]
An object counted by $C_k(T;m)$ is a tree on $[m]$ together with an
unordered $k$-element set of distinct displays of $T$; we call this a
\emph{$k$-collision}. The moment $C_1(T;m)$ is the \emph{marked display
count}. The ordered version is introduced in
Definition~\ref{def:ordered-moment}.
\end{definition}

Fixing the sizes of the contraction blocks separates the choices: each
vertex of $T$ is replaced by a connected block rooted at its survivor, and
each incident target edge chooses an attachment vertex in that block.

\begin{theorem}\label{thm:marked_display}
Let $T\in\T_n$, and let $d_T(i)$ be the degree of vertex $i$ in $T$.  Then
\[
C_1(T;m)=\binom{m}{n}
\sum_{\substack{b_1+\cdots+b_n=m\\b_i\ge1}}
\binom{m-n}{b_1-1,\ldots,b_n-1}
\prod_{i=1}^n b_i^{b_i-2+d_T(i)}.
\]
\end{theorem}

\begin{proof}
Choose the survivor set
\(R=\{r_1<\cdots<r_n\}\), giving \(\binom mn\) choices; the survivor \(r_i\)
is forced to represent target vertex \(i\).  Suppose its block has size
\(b_i\ge1\), with \(\sum_i b_i=m\).  Assigning the remaining \(m-n\)
labels to the ordered blocks contributes
\[
\binom{m-n}{b_1-1,\ldots,b_n-1}
\]
choices.  Cayley's formula~\cite{Cayley} gives \(b_i^{b_i-2}\) internal
trees on the \(i\)th block.  We use the usual value \(1\) when \(b_i=1\).
For every target edge incident to \(i\), choose the endpoint of its
corresponding crossing edge inside the \(i\)th block.  These choices
contribute \(b_i^{d_T(i)}\).

The union of the \(n\) internal block trees and the \(n-1\) selected crossing
edges is connected and has
\[
\sum_{i=1}^n(b_i-1)+(n-1)=m-1
\]
edges, hence is a tree.  Conversely, a marked display recovers its survivor
set, ordered blocks, internal block trees, and the unique crossing edge
representing every target edge.  The construction is therefore bijective.
\end{proof}

\begin{remark}\label{rem:marked-degree-data}
The composition sum is unchanged by permuting its indices, so the marked
count depends on \(T\) only through its degree multiset. In particular,
the first collision moment cannot distinguish targets with the same degree
multiset; distinctions between them can appear only in higher collision
moments.
\end{remark}

\section{Collision moments and bounded cores}\label{sec:cores}
The support count differs from the marked display count exactly when a host
admits more than one display.  For fixed $T$ and $U$, inclusion--exclusion gives
\[
\mathbf 1_{c_T(U)>0}=
\sum_{k\ge1}(-1)^{k+1}\binom{c_T(U)}{k}
\]
for the contribution of $U$ to the support count. Summing over $U$ gives
\[
\mu_T(m)=C_1(T;m)-C_2(T;m)+C_3(T;m)-\cdots.
\]
We now show that contracting every edge contracted by all $k$ displays
leaves only a bounded number of vertices.

\begin{definition}\label{def:ordered-moment}
For $k\ge1$, let $(r)_k=r(r-1)\cdots(r-k+1)$.  The \emph{ordered $k$th collision moment} is
\[
\widehat C_k(T;m)=\sum_{U\in\T_m}(c_T(U))_k.
\]
It counts ordered $k$-tuples of pairwise distinct displays of $T$ in a labeled tree on $[m]$.  Consequently,
\[
\widehat C_k(T;m)=k!\,C_k(T;m).
\]
\end{definition}

\begin{definition}\label{def:overlay}
Let \(T\in\T_n\).  A \emph{\(k\)-overlay state} of \(T\) at order \(m\) is
a tuple
\[
\Xi=(U,\delta_1,\ldots,\delta_k),\qquad U\in\T_m,
\]
where each \(\delta_a\) is a display with display map \(\pi_a\) and survivor
set
\[
r_{a,1}<\cdots<r_{a,n}.
\]
Its \emph{full label map} is
\[
\Pi_\Xi(v)=(\pi_1(v),\ldots,\pi_k(v))\in[n]^k.
\]
An edge \(uv\in E(U)\) is an \emph{equality edge} if
\(\Pi_\Xi(u)=\Pi_\Xi(v)\).  When
\(\delta_1,\ldots,\delta_k\) are pairwise distinct, \(\Xi\) is a
\emph{\(k\)-collision state}.
\end{definition}

\begin{definition}\label{def:reduced-core}
The \emph{reduced map core} \(\rho(\Xi)\) is obtained by contracting every
connected component of the equality-edge subgraph to one vertex.  Each
display map factors through the quotient, and the resulting \(k\) maps are
retained as part of the core.  No edge of a reduced map core is an equality
edge.  The reduction does not record which labels are the survivors;
that information is added in Section~\ref{sec:decorated-cores}.
\end{definition}

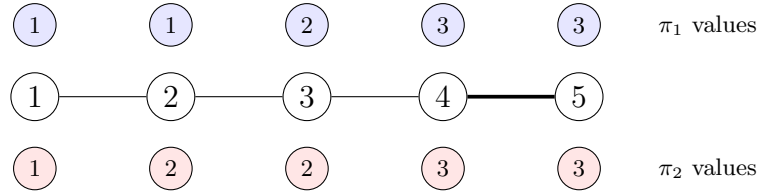
\begin{figure}[H]
\centering
\begin{tikzpicture}[
  scale=1,
  every node/.style={circle,draw,inner sep=1pt,minimum size=18pt},
  topbox/.style={draw,rounded corners=2pt,fill=blue!10,minimum width=16pt,minimum height=13pt,font=\scriptsize},
  bottombox/.style={draw,rounded corners=2pt,fill=red!10,minimum width=16pt,minimum height=13pt,font=\scriptsize},
  lab/.style={draw=none,rectangle,font=\scriptsize}
]
\node (a) at (0,0) {$1$};
\node (b) at (1.8,0) {$2$};
\node (c) at (3.6,0) {$3$};
\node (d) at (5.4,0) {$4$};
\node (e) at (7.2,0) {$5$};
\draw (a)--(b)--(c)--(d);
\draw[line width=1.4pt] (d)--(e);

\node[topbox]    at (0,0.95) {$1$};
\node[topbox]    at (1.8,0.95) {$1$};
\node[topbox]    at (3.6,0.95) {$2$};
\node[topbox]    at (5.4,0.95) {$3$};
\node[topbox]    at (7.2,0.95) {$3$};

\node[bottombox] at (0,-0.95) {$1$};
\node[bottombox] at (1.8,-0.95) {$2$};
\node[bottombox] at (3.6,-0.95) {$2$};
\node[bottombox] at (5.4,-0.95) {$3$};
\node[bottombox] at (7.2,-0.95) {$3$};

\node[lab] at (8.9,0.95) {$\pi_1$ values};
\node[lab] at (8.9,-0.95) {$\pi_2$ values};
\node[lab,align=center] at (3.6,-1.8) {Thus $\Pi_\Xi(4)=\Pi_\Xi(5)=(3,3)$, so the bold edge is an equality edge.};
\end{tikzpicture}
\caption{A pair-overlay state on a labeled tree $U$. The vertex labels inside the white circles are the labels of $U$. The blue row records the first display map $\pi_1$ and the red row records the second display map $\pi_2$. Equality edges are exactly the edges whose endpoints agree in both rows, equivalently whose full pair labels $\Pi_\Xi(v)=(\pi_1(v),\pi_2(v))$ are equal.}
\label{fig:equality}
\end{figure}
\FloatBarrier

\begin{definition}\label{def:crossing-edge}
Let $\pi:U\twoheadrightarrow T$ be a display map.  An edge of $U$ is \emph{$\pi$-crossing} if its endpoints lie in different fibers of $\pi$.
\end{definition}

\begin{lemma}\label{lem:crossing}
Let $\pi:U\twoheadrightarrow T$ be a display map with $T\in\T_n$.  Then exactly $n-1$ edges of $U$ are $\pi$-crossing.
\end{lemma}

\begin{proof}
The fibers of $\pi$ are connected. Contracting them gives the quotient tree
$T$, whose $n-1$ edges correspond exactly to the $\pi$-crossing edges of $U$.
\end{proof}

\begin{theorem}\label{thm:bounded_core}
Let $T\in\T_n$ and let
\[
\Xi=(U,\delta_1,\ldots,\delta_k)
\]
be a $k$-overlay state of $T$.  Then the reduced core $\rho(\Xi)$ has at most
\[
k(n-1)+1
\]
vertices.  In particular, every 2-collision state of an $n$-vertex tree reduces to a reduced map core on at most $2n-1$ vertices.
\end{theorem}

\begin{proof}
An edge that survives the equality-edge contraction is not an equality edge.  Hence its endpoints have different full $k$-tuples, so for at least one display map $\pi_a$ the edge is $\pi_a$-crossing.  By Lemma~\ref{lem:crossing}, each display has exactly $n-1$ crossing edges.  Therefore at most $k(n-1)$ edges can survive the core reduction.  The reduced object is still a tree, because it is obtained from a tree by contractions.  A tree with at most $k(n-1)$ edges has at most $k(n-1)+1$ vertices.
\end{proof}

The reduced map core records how the contraction blocks of the displays
interact, but it does not record their survivors.  This distinction matters
even for a single block: two displays may have the same connected partition
and standardized display map but choose different surviving labels.  We now
restore exactly the data that the map core loses.

\section{Survivor-decorated cores and exact collision sums}
\label{sec:decorated-cores}

The reduced core records the simultaneous quotient maps, but this is not
enough to recover the original displays.  A display includes actual surviving
labels, and two displays with identical contraction blocks can differ in
their choice of survivors.  An exact lift count must
therefore record three additional features: which survivor roles are filled
by the same labeled vertex, in which connected block each survivor lies,
and the order required among the survivors in each display.

Let
\[
    M_{k,n}=[k]\times[n]
\]
be the set of survivor marks; the mark $(a,j)$ denotes the survivor
representing target vertex $j$ in display $a$. A partition of $M_{k,n}$
records the equality pattern of these survivors: two marks lie in the same
part precisely when they are represented by the same surviving host vertex.
Because the display indices and target labels are fixed, the parts of this
partition are named subsets of $M_{k,n}$.

Let $H$ be a tree and let $T\in\mathcal T_n$. Suppose that
$\varphi\colon V(H)\to[n]$ is surjective and that every fiber
$\varphi^{-1}(j)$ induces a connected subtree of $H$. Define the quotient
$H/\varphi$ to be the graph on $[n]$ in which distinct vertices $i$ and $j$
are adjacent precisely when some edge of $H$ has one endpoint in
$\varphi^{-1}(i)$ and the other in $\varphi^{-1}(j)$. Equivalently,
$H/\varphi$ is obtained by contracting every fiber of $\varphi$ to one
vertex. We call $\varphi$ a $T$-quotient map if
\[
    H/\varphi=T.
\]
This notion contains no survivor data. The vertices of $H$ need not be
ordered or labeled, and no vertex of a fiber is selected to survive.

\begin{definition}
An ordered survivor-decorated $k$-core for $T\in\mathcal T_n$ is a tuple
\[
    \Omega=(H,\varphi_1,\ldots,\varphi_k,P,\beta)
\]
with the following data and compatibility conditions.

\begin{enumerate}
\item[\textnormal{(C1)}] For every $a\in[k]$, the map
$\varphi_a\colon V(H)\to[n]$ is a $T$-quotient map.

\item[\textnormal{(C2)}] The map
$\Phi=(\varphi_1,\ldots,\varphi_k)$ is reduced. For every
$xy\in E(H)$, one has $\Phi(x)\ne\Phi(y)$.

\item[\textnormal{(S1)}] $P$ is a partition of $M_{k,n}$. Write
$C_{a,j}\in P$ for the part containing $(a,j)$. A single surviving host
vertex is assigned to all marks in one part, and distinct parts are assigned
distinct surviving host vertices. Thus, in every realization, the equality
convention is
\[
    C_{a,j}=C_{b,\ell}
    \quad\Longleftrightarrow\quad
    r_{a,j}=r_{b,\ell}.
\]

\item[\textnormal{(S2)}] The map $\beta\colon P\to V(H)$ records the core
block containing the survivor represented by each class. For every
$(a,j)\in M_{k,n}$,
\[
    \varphi_a\bigl(\beta(C_{a,j})\bigr)=j.
\]
The map $\beta$ need not be injective, since distinct survivor vertices may lie
in the same core block.

\item[\textnormal{(S3)}] For every $a\in[k]$, the classes
$C_{a,1},\ldots,C_{a,n}$ are pairwise distinct. Moreover, the directed
relation on $P$ generated by
\[
    C_{a,1}<C_{a,2}<\cdots<C_{a,n}
    \qquad(a\in[k])
\]
is acyclic. Equivalently, these required inequalities extend to a strict
linear order on $P$.

\item[\textnormal{(D)}] The encoded displays are pairwise distinct.
Whenever $a\ne b$, either $\varphi_a\ne\varphi_b$ as maps or
$C_{a,j}\ne C_{b,j}$ for at least one $j\in[n]$.
\end{enumerate}
\end{definition}

Two decorated cores
\[
  \Omega=(H,\varphi_1,\ldots,\varphi_k,P,\beta),\qquad
  \Omega'=(H',\varphi'_1,\ldots,\varphi'_k,P',\beta')
\]
are isomorphic only if $P=P'$ as partitions of the named mark set
$M_{k,n}$. When $P=P'$, an isomorphism
$g\colon\Omega\to\Omega'$ is a tree isomorphism $g\colon H\to H'$
satisfying
\[
    \varphi'_a\circ g=\varphi_a\quad(a\in[k]),
    \qquad
    g\bigl(\beta(C)\bigr)=\beta'(C)\quad(C\in P).
\]
In particular, neither the ordered display indices nor the survivor classes
may be permuted. Hence
\[
  \operatorname{Aut}(\Omega)=
  \left\{
      g\in\operatorname{Aut}(H):
      \varphi_a\circ g=\varphi_a\ \text{for every }a\in[k],\quad
      g\bigl(\beta(C)\bigr)=\beta(C)\ \text{for every }C\in P
  \right\}.
\]

Condition \textnormal{(C1)} records the connected quotient blocks of each
display but does not select their survivors. Those choices are recorded by
$P$, $\beta$, and later by $\theta$. Condition \textnormal{(C2)} ensures
that no edge of $H$ is contracted by all $k$ quotient maps simultaneously.
Conditions \textnormal{(S1)} and \textnormal{(S2)} record the equality
pattern and core locations of the survivors. Condition \textnormal{(S3)}
ensures that their required numerical orders are compatible. Condition
\textnormal{(D)} excludes repetitions in the ordered tuple of displays.

\begin{lemma}
\label{lem:finite-decorated-cores}
For fixed \(T\in\T_n\) and \(k\ge1\), there are only finitely many
isomorphism classes of survivor-decorated \(k\)-cores with at most
\(k(n-1)+1\) vertices.
\end{lemma}

\begin{proof}
There are finitely many trees of bounded order, finitely many \(k\)-tuples of
maps from each such vertex set to \([n]\), finitely many partitions of the
finite set \(M_{k,n}\), and finitely many maps from the parts of any
such partition to the core vertices.  The defining conditions select a
subset of this finite collection.
\end{proof}

Fix a representative decorated core $\Omega$. For a family of nonempty,
pairwise disjoint sets $(B_x)_{x\in V(H)}$ whose union is $[m]$, define
\begin{equation}\label{eq:survivor-assignment}
  N_\Omega(\mathbf B)=
  \#\left\{
    \begin{array}{@{}l|l@{}}
      \theta\colon P\hookrightarrow[m]
      & \theta(C)\in B_{\beta(C)}\quad(C\in P),\\[-1mm]
      & \theta(C_{a,1})<\cdots<\theta(C_{a,n})
        \quad(a\in[k])
    \end{array}
  \right\}.
\end{equation}
Thus $\theta$ assigns an actual host label to each survivor class. On setting
\[
    r_{a,j}=\theta(C_{a,j}),
\]
injectivity gives the exact equality pattern
\[
    r_{a,j}=r_{b,\ell}
    \quad\Longleftrightarrow\quad
    C_{a,j}=C_{b,\ell},
\]
while the inequalities in \eqref{eq:survivor-assignment} enforce increasing
standardization in every display.
\begin{definition}\label{def:rigid-lift}
The rigidified lift count of $\Omega$ to $[m]$ is
\begin{equation}
 L_m^{\mathrm{rig}}(\Omega)=
 \sum_{\substack{[m]=\bigsqcup_{x\in V(H)}B_x\\B_x\ne\varnothing}}
 N_\Omega(\mathbf B)
 \prod_{x\in V(H)}
 |B_x|^{|B_x|-2+d_H(x)}.
\label{eq:rigid-lift}
\end{equation}
The sum is over blocks indexed by the vertices of the chosen representative
$H$.  We use the Cayley convention that the tree factor for a singleton
block is $1$.
\end{definition}

Here the factor \(|B_x|^{|B_x|-2}\) chooses the internal Cayley tree in the
block \(B_x\), and \(|B_x|^{d_H(x)}\) chooses one attachment vertex for each
core-edge incidence at \(x\).  The value for a singleton block is \(1\).
The lift is called \emph{rigidified} because its equality components are
identified with the named vertices of the chosen representative \(H\).
Accordingly, the blocks in \eqref{eq:rigid-lift} are indexed by \(V(H)\)
rather than taken only up to a core automorphism.

\begin{theorem}\label{thm:decorated-lifts}
Let $T\in\mathcal T_n$ and $k\ge 1$. Let $\mathcal C_k(T)$ be the finite set
of isomorphism classes of ordered survivor-decorated $k$-cores for $T$ whose
underlying tree has at most $k(n-1)+1$ vertices; its finiteness follows from
Lemma~\ref{lem:finite-decorated-cores}. Then
\begin{equation}\label{eq:decorated-core-formula}
  \widehat C_k(T;m)
    =\sum_{[\Omega]\in\mathcal C_k(T)}
       \frac{L_m^{\mathrm{rig}}(\Omega)}
            {|\operatorname{Aut}(\Omega)|},
  \qquad
  C_k(T;m)=\frac{\widehat C_k(T;m)}{k!}.
\end{equation}
\end{theorem}

\begin{proof}
Let
\[
    \Xi=(U,\delta_1,\ldots,\delta_k)
\]
be an ordered $k$-collision state on $[m]$. Write $\pi_a$ for the display
map of $\delta_a$ and
\[
    r_{a,1}<\cdots<r_{a,n}
\]
for its survivors. In particular,
\[
    \pi_a(r_{a,j})=j.
\]

Let $H_\Xi$ be the underlying tree of the reduced map core $\rho(\Xi)$, and
let
\[
    q_0\colon V(U)\longrightarrow V(H_\Xi)
\]
be the equality-edge quotient. Every map $\pi_a$ is constant on the fibers
of $q_0$. It therefore factors uniquely as
\[
    \pi_a=\varphi_a^\Xi\circ q_0
\]
for a map
\[
    \varphi_a^\Xi\colon V(H_\Xi)\longrightarrow[n].
\]

We verify that $\varphi_a^\Xi$ is a $T$-quotient map. It is surjective
because $\pi_a$ is surjective. For every $j\in[n]$,
\[
    q_0^{-1}\bigl((\varphi_a^\Xi)^{-1}(j)\bigr)
       =\pi_a^{-1}(j).
\]
The set on the right induces a connected subtree of $U$ because $\pi_a$
is a display map. Its image under $q_0$ is consequently connected in
$H_\Xi$. Thus every fiber of $\varphi_a^\Xi$ is connected. Contracting
these fibers in $H_\Xi$ has the same effect as contracting the fibers of
$\pi_a$ in $U$. The resulting quotient is $T$. Hence
$\varphi_a^\Xi$ is a $T$-quotient map.

Define a partition $P_\Xi$ of $M_{k,n}$ by
\[
    (a,j)\sim_\Xi(b,\ell)
    \quad\Longleftrightarrow\quad
    r_{a,j}=r_{b,\ell}.
\]
Let $C_{a,j}^\Xi$ be the class containing $(a,j)$. Define
\[
    \beta_\Xi(C_{a,j}^\Xi)=q_0(r_{a,j}).
\]
This is well defined because marks in the same class have the same
survivor. It also satisfies
\[
    \varphi_a^\Xi\bigl(\beta_\Xi(C_{a,j}^\Xi)\bigr)
       =\pi_a(r_{a,j})
       =j.
\]

For each $a$, the classes
\[
    C_{a,1}^\Xi,\ldots,C_{a,n}^\Xi
\]
are distinct. The relations
\[
    C_{a,1}^\Xi<\cdots<C_{a,n}^\Xi
\]
are acyclic because they are realized by the strict inequalities
\[
    r_{a,1}<\cdots<r_{a,n}.
\]
Condition \textnormal{(D)} also holds. If it failed for some $a\ne b$,
then
\[
    \varphi_a^\Xi=\varphi_b^\Xi
    \qquad\text{and}\qquad
    C_{a,j}^\Xi=C_{b,j}^\Xi
    \quad\text{for every }j.
\]
The first equality would give $\pi_a=\pi_b$, while the second would give
$r_{a,j}=r_{b,j}$ for every $j$. For a fixed display map and ordered
survivors, the destination map is given by
\[
    v\longmapsto r_{a,\pi_a(v)}.
\]
Thus the two destination maps would also be equal. This would imply
$\delta_a=\delta_b$, contrary to the definition of a collision state.

We have obtained a survivor-decorated core
\[
    \Omega_\Xi=
    (H_\Xi,\varphi_1^\Xi,\ldots,\varphi_k^\Xi,P_\Xi,\beta_\Xi).
\]
It satisfies \textnormal{(C2)} because no edge remaining after the
equality-edge contraction is an equality edge. Theorem~\ref{thm:bounded_core} gives
\[
    |V(H_\Xi)|\leq k(n-1)+1.
\]

Fix a representative
\[
    \Omega=(H,\varphi_1,\ldots,\varphi_k,P,\beta)
\]
of an isomorphism class in $\mathcal C_k(T)$. A rigidification of a
collision state $\Xi$ whose decorated core is isomorphic to $\Omega$ is
an isomorphism
\[
    \iota\colon\Omega_\Xi\longrightarrow\Omega.
\]
We show that $L_m^{\mathrm{rig}}(\Omega)$ counts pairs $(\Xi,\iota)$ of
this form.

Take such a pair. Write
\[
    \iota\colon H_\Xi\longrightarrow H
\]
also for the underlying tree isomorphism, and put
\[
    q=\iota\circ q_0.
\]
For every $x\in V(H)$, let
\[
    B_x=q^{-1}(x).
\]
The sets $B_x$ are nonempty and connected, and they partition $[m]$.

An isomorphism of decorated cores preserves the named partition, so
$P_\Xi=P$. Define
\[
    \theta(C_{a,j})=r_{a,j}.
\]
This is well defined because marks in the same class have the same
survivor. It is injective because marks in different classes have
different survivors. Since $\iota$ preserves $\beta$, we have
\[
    \theta(C)\in B_{\beta(C)}
    \qquad(C\in P).
\]
The survivor order gives
\[
    \theta(C_{a,1})<\cdots<\theta(C_{a,n})
    \qquad(a\in[k]).
\]
Thus $\theta$ is counted by $N_\Omega(\mathbf B)$.

The edges of $U$ inside $B_x$ form a tree on $B_x$. For every
$xy\in E(H)$, exactly one edge of $U$ joins $B_x$ to $B_y$. There cannot
be two such edges because the blocks are connected and $U$ is a tree.
The endpoints of this edge give one attachment choice in each block.
The internal trees and attachment choices contribute
\[
    \prod_{x\in V(H)}
       |B_x|^{\,|B_x|-2+d_H(x)}.
\]
The pair $(\Xi,\iota)$ therefore determines one choice counted by
$L_m^{\mathrm{rig}}(\Omega)$.

Conversely, fix a choice counted by $L_m^{\mathrm{rig}}(\Omega)$. It
consists of an indexed partition
\[
    [m]=\coprod_{x\in V(H)}B_x,
\]
an injection $\theta$ counted by $N_\Omega(\mathbf B)$, a tree on each
$B_x$, and an attachment vertex in $B_x$ for every incidence of a core
edge at $x$.

For each $xy\in E(H)$, join the selected vertex in $B_x$ to the selected
vertex in $B_y$. The resulting graph $U$ is connected and has
\[
    \sum_{x\in V(H)}(|B_x|-1)+|E(H)|=m-1
\]
edges. Hence $U$ is a tree. Let
\[
    q\colon V(U)\longrightarrow V(H)
\]
be the block map and set
\[
    \pi_a=\varphi_a\circ q.
\]

For every $a\in[k]$ and $j\in[n]$, define
\[
    r_{a,j}=\theta(C_{a,j})
\]
and
\[
    R_a=\{r_{a,1},\ldots,r_{a,n}\}.
\]
The inequalities in the definition of $N_\Omega(\mathbf B)$ give
\[
    r_{a,1}<\cdots<r_{a,n}.
\]
Define
\[
    d_a\colon[m]\longrightarrow R_a,
    \qquad
    d_a(v)=r_{a,\pi_a(v)}.
\]

We now check Definition~\ref{def:display}. Since
\[
    r_{a,j}=\theta(C_{a,j})\in B_{\beta(C_{a,j})},
\]
we have
\[
    q(r_{a,j})=\beta(C_{a,j}).
\]
Condition \textnormal{(S2)} gives
\[
    \pi_a(r_{a,j})
       =\varphi_a\bigl(q(r_{a,j})\bigr)
       =\varphi_a\bigl(\beta(C_{a,j})\bigr)
       =j.
\]
It follows that
\[
    d_a(r_{a,j})=r_{a,j},
\]
so $d_a$ is a retraction onto $R_a$.

The fiber over $r_{a,j}$ is
\[
    d_a^{-1}(r_{a,j})
       =\pi_a^{-1}(j)
       =\bigcup_{x\in\varphi_a^{-1}(j)}B_x.
\]
Condition \textnormal{(C1)} says that $\varphi_a^{-1}(j)$ induces a
connected subtree of $H$. Each $B_x$ is connected, and the construction
places a joining edge between $B_x$ and $B_y$ whenever $xy$ is an edge
of this subtree. Therefore $d_a^{-1}(r_{a,j})$ induces a connected
subtree of $U$.

Contracting the fibers of $d_a$ first contracts the blocks $B_x$ and
produces $H$. It then contracts the fibers of $\varphi_a$. The resulting
quotient is
\[
    H/\varphi_a=T
\]
by \textnormal{(C1)}. Relabeling $r_{a,j}$ by $j$ therefore gives exactly
$T$. Hence
\[
    \delta_a=(R_a,d_a)
\]
is a display of $T$ in $U$. Its display map is $\pi_a$, since
\[
    \operatorname{st}_{R_a}(d_a(v))
       =\operatorname{st}_{R_a}(r_{a,\pi_a(v)})
       =\pi_a(v).
\]

These displays are pairwise distinct. Suppose that
$\delta_a=\delta_b$. Their display maps are then equal, so
\[
    \varphi_a\circ q=\varphi_b\circ q.
\]
Since $q$ is surjective,
\[
    \varphi_a=\varphi_b.
\]
The ordered survivor sets are also equal, which gives
\[
    r_{a,j}=r_{b,j}
    \qquad(j\in[n]).
\]
Injectivity of $\theta$ now gives
\[
    C_{a,j}=C_{b,j}
    \qquad(j\in[n]).
\]
This contradicts condition \textnormal{(D)}.

We next identify the equality-edge reduction of the constructed state.
Put
\[
    \Pi=(\pi_1,\ldots,\pi_k)
    \qquad\text{and}\qquad
    \Phi=(\varphi_1,\ldots,\varphi_k).
\]
Then
\[
    \Pi=\Phi\circ q.
\]
Every edge inside a block $B_x$ is an equality edge. An edge joining
$B_x$ to $B_y$ corresponds to an edge $xy\in E(H)$. Condition
\textnormal{(C2)} gives
\[
    \Phi(x)\ne\Phi(y),
\]
so no joining edge is an equality edge. Since each $B_x$ is connected,
the sets $B_x$ are exactly the connected components of the equality-edge
subgraph. The reduced map core is therefore $H$, not a proper quotient
of $H$.

The survivor partition recovered from the constructed displays is $P$.
Indeed, injectivity of $\theta$ gives
\[
    r_{a,j}=r_{b,\ell}
    \quad\Longleftrightarrow\quad
    \theta(C_{a,j})=\theta(C_{b,\ell})
    \quad\Longleftrightarrow\quad
    C_{a,j}=C_{b,\ell}.
\]
The recovered location map is $\beta$ because
\[
    q(r_{a,j})=\beta(C_{a,j}).
\]
Thus the survivor-decorated core of the constructed state is $\Omega$.
The identity map on $H$ supplies its rigidification.

Starting with a pair $(\Xi,\iota)$, the first construction records the
fibers of $\iota\circ q_0$, the trees induced on those fibers, the unique
edges between adjacent fibers, and the actual survivor labels. The reverse
construction restores the same host tree, display maps, destination maps,
survivors, and rigidification. Conversely, starting with a choice counted
by $L_m^{\mathrm{rig}}(\Omega)$, performing the construction and then
extracting these data recovers the original blocks, internal trees,
attachment vertices, and survivor assignment. The two constructions are
therefore inverse.

Consequently, $L_m^{\mathrm{rig}}(\Omega)$ is the number of pairs
$(\Xi,\iota)$ in which $\Xi$ is an ordered $k$-collision state and
\[
    \iota\colon\Omega_\Xi\longrightarrow\Omega
\]
is a rigidification.

Now fix an ordered collision state $\Xi$ with
$\Omega_\Xi\cong\Omega$, and choose one rigidification
\[
    \iota_0\colon\Omega_\Xi\longrightarrow\Omega.
\]
Composition with $\iota_0$ gives a map from
$\operatorname{Aut}(\Omega)$ to the set of rigidifications of $\Xi$,
\[
    g\longmapsto g\circ\iota_0.
\]
This map is injective. Indeed,
$g_1\circ\iota_0=g_2\circ\iota_0$ implies $g_1=g_2$ because
$\iota_0$ is surjective. It is also surjective. If
\[
    \iota\colon\Omega_\Xi\longrightarrow\Omega
\]
is another rigidification, then
\[
    g=\iota\circ\iota_0^{-1}
\]
belongs to $\operatorname{Aut}(\Omega)$ and satisfies
$\iota=g\circ\iota_0$. Hence every collision state whose decorated core
is isomorphic to $\Omega$ has exactly
\[
    |\operatorname{Aut}(\Omega)|
\]
rigidifications.

It follows that the number of ordered collision states whose
survivor-decorated core is isomorphic to $\Omega$ is
\[
    \frac{L_m^{\mathrm{rig}}(\Omega)}
         {|\operatorname{Aut}(\Omega)|}.
\]
Every ordered collision state determines one isomorphism class in
$\mathcal C_k(T)$. Summing over these classes proves the first identity
in~\eqref{eq:decorated-core-formula}. Finally, every unordered $k$-element set of distinct
displays has exactly $k!$ orderings, which proves the second identity.
\end{proof}
We next recover the one-display formula by taking $k=1$.
When $k=1$, there is only one survivor-decorated core type that contributes.
Indeed, the equality edges are exactly the edges lying inside the fibers of
the display map. Contracting their connected components gives $T$ itself.
Thus the unique core has
\[
    H=T,\qquad \varphi_1=\operatorname{id}_{[n]},
\]
where the vertices of $H$ are identified with $[n]$ through $\varphi_1$.
Its survivor partition and location map are
\[
    P=\bigl\{\{(1,1)\},\ldots,\{(1,n)\}\bigr\},
    \qquad
    \beta(\{(1,j)\})=j.
\]
Any automorphism of this decorated core must preserve $\varphi_1$ and
therefore fixes every vertex. Hence
\[
    |\operatorname{Aut}(\Omega)|=1.
\]

For an indexed partition
\[
    [m]=B_1\coprod\cdots\coprod B_n,
\]
the survivor factor becomes
\[
    N_\Omega(\mathbf B)
      =
    \#\left\{
       (r_1,\ldots,r_n):
       r_j\in B_j,\ 
       r_1<\cdots<r_n
     \right\}.
\]
Fix block sizes $|B_j|=b_j$. Summing this quantity over all indexed
partitions of the prescribed sizes gives
\[
    \sum_{\substack{
       [m]=B_1\coprod\cdots\coprod B_n\\
       |B_j|=b_j}}
       N_\Omega(\mathbf B)
    =
    \binom{m}{n}
    \binom{m-n}{b_1-1,\ldots,b_n-1}.
\]
To see this, first choose the survivor set
\[
    \{r_1<\cdots<r_n\},
\]
and then distribute the remaining $m-n$ labels among the blocks so that
$B_j$ receives $b_j-1$ additional labels.

The underlying core is $T$, so
\[
    d_H(j)=d_T(j).
\]
Equation~\eqref{eq:rigid-lift} therefore gives
\[
\begin{aligned}
    L_m^{\mathrm{rig}}(\Omega)
      &=
    \binom{m}{n}
    \sum_{\substack{
       b_1+\cdots+b_n=m\\
       b_j\geq1}}
    \binom{m-n}{b_1-1,\ldots,b_n-1}
    \prod_{j=1}^n
       b_j^{\,b_j-2+d_T(j)}.
\end{aligned}
\]
Since the automorphism group is trivial and
$\widehat C_1(T;m)=C_1(T;m)$, Theorem~\ref{thm:decorated-lifts} reduces exactly to the formula
in Theorem~\ref{thm:marked_display}.

\paragraph{A two-display example.}
Let $T$ be the tree on $[2]$ with edge $12$, and take $m=3$. There are
three trees on $[3]$,
\[
    E(U_1)=\{12,13\},\qquad
    E(U_2)=\{12,23\},\qquad
    E(U_3)=\{13,23\}.
\]
Each of them has four displays of $T$. Consequently, direct enumeration
gives
\[
    \widehat C_2(T;3)
      =
    \sum_{U\in\mathcal T_3}(c_T(U))_2
      =
    3(4)_2
      =
    36,
\]
and hence
\[
    C_2(T;3)=18.
\]

The complete list of the eighteen survivor-decorated core types
contributing to Equation~\eqref{eq:decorated-core-formula}, together with their rigidified
lift counts, automorphism factors, and direct multiplicities, is given in
Appendix~\ref{app:two-display-enumeration}. Every automorphism group in this
example is trivial, and the computation gives
\[
    \sum_{\substack{
       [\Omega]\in\mathcal C_2(T)\\
       L_3^{\mathrm{rig}}(\Omega)>0}}
    \frac{L_3^{\mathrm{rig}}(\Omega)}
         {|\operatorname{Aut}(\Omega)|}
    =36
    =\widehat C_2(T;3),
\]
in agreement with the direct enumeration above.

\subsection*{Information in the second collision moment}
The second collision moment can distinguish targets that the first collision
moment cannot. Let
\[
E(T_1)=\{12,13,14\},
\qquad
E(T_2)=\{12,23,24\}.
\]
Thus $T_1$ and $T_2$ are stars on $[4]$, centered at $1$ and $2$,
respectively. They have the same degree multiset
\[
\{3,1,1,1\},
\]
and they are not reversals of one another. Since the composition sum in
Theorem~\ref{thm:marked_display} is unchanged by simultaneously permuting its indices, we have
\[
C_1(T_1;m)=C_1(T_2;m)
\qquad (m\geq 4).
\]

The second moments already differ at $m=5$. The Pr\"ufer-word enumeration and
the full display distributions are given in
Appendix~\ref{app:second-moment-enumeration}. They give
\[
C_1(T_1;5)=C_1(T_2;5)=70,
\]
whereas
\[
C_2(T_1;5)=65,
\qquad
C_2(T_2;5)=55.
\]
In fact, the corresponding support counts are
\[
\mu_{T_1}(5)=32,
\qquad
\mu_{T_2}(5)=36.
\]
Thus pair collisions detect information about the positions of the labels
that is invisible to the degree multiset and therefore to $C_1$. This example
does not show that $C_2$ determines the target, but it does show that the
second collision moment distinguishes some targets that $C_1$ cannot.

\begin{remark}[Why the survivor decoration is necessary]
\label{rem:survivor-necessary}
The reduced maps and core degrees do not record the survivor choices.
For example, a display of a three-vertex path may have a contraction block
containing two host labels.  Choosing either label as the survivor can give
two distinct displays with the same connected partition and standardized
display map.  The factor \(N_\Omega(\mathbf B)\) records exactly this missing
information.  Consequently, the full collision moments cannot in general be
recovered from the reduced maps and core degrees alone.
\end{remark}

\section{Asymptotic containment}\label{sec:asymptotic}
The support sequence is meant to distinguish trees from each other, so it is natural to
ask whether their containment densities remain separated as the host grows.
This section shows that they do not. Every fixed target occurs in almost every sufficiently large
labeled tree.  It is enough to count the stronger event that the host
contains $T$ exactly on some vertex set.  Moon's forest-counting
formula~\cite{Moon} supplies the single-event probability, and a dependency
estimate controls the overlap among these events.

\begin{theorem}\label{thm:moon}
Let $F$ be a forest on $[m]$ with connected components of sizes $q_1,\ldots,q_r$.  The number of labeled trees on $[m]$ containing $F$ is
\[
q_1q_2\cdots q_r\,m^{r-2}.
\]
\end{theorem}

\begin{proof}
Contract each component of $F$ to a single vertex.  If the component tree on the $r$ vertices has degrees $d_1,\ldots,d_r$, then the number of ways to realize its intercomponent edges is $\prod_i q_i^{d_i}$.  Summing over all labeled trees on the $r$ components and using the weighted Cayley identity
\[
\sum_{R\in\T_r}\prod_{i=1}^r x_i^{d_R(i)}=(x_1+\cdots+x_r)^{r-2}\prod_{i=1}^r x_i,
\]
with $x_i=q_i$, gives $m^{r-2}\prod_i q_i$.
\end{proof}

\begin{definition}\label{def:exact-copy}
For each $n$-subset $S=\{s_1<\cdots<s_n\}\subseteq[m]$, let $T_S$ be the copy of $T$ obtained by replacing label $i$ by $s_i$.  For a uniformly random labeled tree $U_m$ on $[m]$, let $I_S(U_m)$ be the indicator of the event $U_m[S]=T_S$, and set
\[
Y_T(U_m)=\sum_{S\in\binom{[m]}{n}} I_S(U_m).
\]
We call $T_S$ an \emph{exact standardized copy} of $T$ on $S$.  If $Y_T(U_m)>0$, then $U_m$ displays $T$.
\end{definition}

\begin{lemma}\label{lem:single_prob}
For every fixed $S\in\binom{[m]}n$,
\[
\mathbb P(I_S=1)=\frac{n}{m^{n-1}}.
\]
Consequently,
\[
\lambda_m:=\mathbb E Y_T(U_m)=\binom mn\frac{n}{m^{n-1}}=\frac{m}{(n-1)!}+O_T(1).
\]
\end{lemma}

\begin{proof}
The event $I_S=1$ means that $U_m$ contains the prescribed tree $T_S$ on $S$.  Since $U_m$ is itself a tree, there can be no additional edge among vertices of $S$.  Apply Theorem~\ref{thm:moon} to the forest consisting of one component $T_S$ of size $n$ and $m-n$ isolated vertices.  The number of containing trees is $n m^{m-n-1}$.  Dividing by the total number $m^{m-2}$ of labeled trees gives $n/m^{n-1}$.
\end{proof}

\begin{lemma}\label{lem:dependency}
Put a graph on the set $\binom{[m]}n$ by joining distinct $S$ and $R$ when $S\cap R\ne\varnothing$.  This is a dependency graph for the events $I_S=1$.
\end{lemma}

\begin{proof}
Fix $S$, and let $R_1,\ldots,R_t$ be any collection of $n$-sets
disjoint from $S$; the sets $R_i$ need not be disjoint from one another.
For $J\subseteq[t]$, consider the prescribed edge set
\[
 F_J=\bigcup_{j\in J}T_{R_j}.
\]
If the prescribed copies are incompatible or their union contains a cycle, then both
$\bigcap_{j\in J}\{I_{R_j}=1\}$ and its intersection with
$\{I_S=1\}$ are empty, so the required identity holds trivially. Otherwise $F_J$ is a forest.  Since every $R_j$ is
disjoint from $S$, the tree $T_S$ is a new component disjoint from $F_J$.
Moon's formula then shows that adjoining $T_S$ multiplies the probability
that a uniform tree contains the prescribed forest by exactly
\[
 \frac{n}{m^{n-1}}=\mathbb P(I_S=1).
\]
Indeed, the new component contributes the factor $n$ and reduces the number
of forest components by $n-1$.

Consequently, for every $J\subseteq[t]$,
\[
 \mathbb P\left(I_S=1,\ I_{R_j}=1\ (j\in J)\right)
 =\mathbb P(I_S=1)
  \mathbb P\left(I_{R_j}=1\ (j\in J)\right).
\]
By inclusion--exclusion, the same identity holds when any of the variables
is required to vanish.  Thus $I_S$ is independent of the $\sigma$-algebra
generated by the variables $I_R$ with $R\cap S=\varnothing$.
The same argument applies to two collections of subsets whose unions are
disjoint: Moon's formula factors over the two prescribed forests, and
inclusion--exclusion gives independence of the two families of indicators.
Thus the graph is a dependency graph.
\end{proof}

\begin{theorem}\label{thm:suen}
Let $\{I_\alpha\}_{\alpha\in\mathcal I}$ be indicator random variables with dependency graph $\Gamma$.  Let $X=\sum_\alpha I_\alpha$, $\lambda=\mathbb E X$, and define
\[
\Delta=\frac12\sum_\alpha\sum_{\beta\sim\alpha}
\mathbb E[I_\alpha I_\beta],\qquad
\delta=\max_\alpha\sum_{\beta\sim\alpha}\mathbb E[I_\beta].
\]
Then
\[
\mathbb P(X=0)\le
\exp\left(
-\min\left\{
\frac{\lambda^2}{8\Delta},
\frac{\lambda}{2},
\frac{\lambda}{6\delta}
\right\}\right),
\]
where a fraction with zero denominator is interpreted as \(+\infty\).
\end{theorem}

\begin{proof}
This is a standard dependency-graph form of Suen's correlation inequality;
see Janson~\cite{Janson}, Alon and Spencer~\cite[Section 8.7]{AlonSpencer},
and Suen~\cite{Suen}.
\end{proof}

\begin{lemma}\label{lem:suen_bounds}
For fixed $T\in\T_n$, there are constants $a_T,b_T,d_T>0$ such that
\[
\lambda=\mathbb E Y_T(U_m)\ge a_Tm,
\qquad
\Delta\le b_Tm,
\qquad
\delta\le d_T.
\]
\end{lemma}

\begin{proof}
The lower bound for $\lambda$ follows from Lemma~\ref{lem:single_prob}.
For \(\Delta\), it is enough to sum over ordered adjacent pairs
\(S\sim R\) with \(S\ne R\), since the factor \(1/2\) can only improve the
upper bound.
If $|S\cap R|=q$, where $1\le q\le n-1$, then there are
$O_T(m^{2n-q})$ such ordered pairs.  If the joint event is possible, the two
prescribed copies overlap compatibly.  Their union is therefore a connected
forest, hence a tree, on $2n-q$ vertices.  Moon's formula gives joint
probability $O_T(m^{-(2n-q-1)})$.  Thus each fixed $q$ contributes
$O_T(m)$, and summing over $q$ gives $\Delta=O_T(m)$.

For $\delta$, fix $S$.  The number of $R$ with $|S\cap R|=q$ is $O_T(m^{n-q})$, and by Lemma~\ref{lem:single_prob}, $\mathbb E I_R=n/m^{n-1}$.  Summing over $q\ge1$ gives
\[
\sum_{R\sim S}\mathbb E I_R=O_T\left(\sum_{q=1}^{n-1}m^{n-q}m^{-(n-1)}\right)=O_T(1).\qedhere
\]
\end{proof}

\begin{theorem}\label{thm:exp_containment}
For every fixed labeled tree $T\in\T_n$, there exists $c_T>0$ such that
\[
0\le 1-\frac{\mu_T(m)}{m^{m-2}}=O_T(e^{-c_Tm}).
\]
Equivalently,
\[
m^{m-2}-\mu_T(m)=O_T(m^{m-2}e^{-c_Tm}).
\]
\end{theorem}

\begin{proof}
Apply Theorem~\ref{thm:suen} to $Y_T$.  By Lemma~\ref{lem:suen_bounds}, for suitable constants $a_T,b_T,d_T>0$,
\[
\lambda\ge a_Tm,
\qquad
\Delta\le b_Tm,
\qquad
\delta\le d_T.
\]
Consequently,
\[
\frac{\lambda^2}{8\Delta}\ge \frac{a_T^2}{8b_T}m,
\qquad
\frac{\lambda}{2}\ge\frac{a_T}{2}m,
\qquad
\frac{\lambda}{6\delta}\ge\frac{a_T}{6d_T}m.
\]
Suen's inequality gives
\(\mathbb P(Y_T=0)\le e^{-c_Tm}\) after decreasing \(c_T>0\) if necessary.
If \(U_m\) has no display of \(T\), then \(Y_T(U_m)=0\).  Hence
\[
\mathbb P(T\npreceq U_m)
\le \mathbb P(Y_T=0)
\le e^{-c_Tm}.
\]
Multiplying by \(|\T_m|=m^{m-2}\) proves the result.
\end{proof}

\section{Top-centered star avoidance}\label{sec:topstar}
Fix $n\geq3$, and let $S_n$ be the star on $[n+1]$ centered at $n+1$.
Every display of $S_n$ has a center survivor that is larger than its $n$
leaf survivors. This follows from increasing standardization because the
center represents the target label $n+1$. The center survivor need not be
the largest label of the host. We study the stronger condition that the
absolute largest host label $k$ survives as the center. Under this
restriction, containment is determined by the number of host leaves and
whether $k$ is a leaf.

\begin{definition}
Let $U\in\mathcal T_k$. A display of $S_n$ in $U$ is top-centered if the
survivor representing the center $n+1$ is the largest host label
$k$. Let $A_n^{\mathrm{top}}(k)$ be the number of trees
$U\in\mathcal T_k$ containing no top-centered display of $S_n$.
\end{definition}

\begin{definition}\label{def:leaf-count}
For $k\ge2$, let $\tau(k,L)$ be the number of labeled trees on $[k]$ with exactly $L$ leaves.  Impossible parameter values contribute zero.
\end{definition}

\begin{lemma}\label{lem:tau}
For $k\ge2$,
\[
\tau(k,L)=\binom{k}{L}(k-L)!\stir{k-2}{k-L}.
\]
\end{lemma}

\begin{proof}
In the Pr\"ufer code of a labeled tree on $[k]$, the leaves are exactly the labels that do not occur.  Choose the $L$ absent labels.  The Pr\"ufer word has length $k-2$ and must use every one of the remaining $k-L$ labels at least once.  The number of such words is $(k-L)!\stir{k-2}{k-L}$, by the standard interpretation of Stirling numbers of the second kind; see, for example, Stanley~\cite[Section~1.9]{Stanley1}.  Multiplying by $\binom{k}{L}$ proves the formula.
\end{proof}

\begin{lemma}\label{lem:leafcriterion}
Let $U\in\T_k$ have $L$ leaves.  If $L<n$, then $U$ avoids top-centered $S_n$.  If $L>n$, then $U$ contains a top-centered $S_n$.  If $L=n$, then $U$ contains a top-centered $S_n$ if and only if the largest label $k$ is not a leaf.
\end{lemma}

\begin{proof}
Every leaf block of a tree display contains a leaf of the original tree.
Indeed, such a block is connected and has exactly one edge to its complement.
Starting at the endpoint of that boundary edge and moving within the block as
far as possible produces a vertex of degree one in \(U\).  Distinct quotient
leaf blocks contain distinct original leaves.  A quotient therefore cannot
have more leaves than its host, proving avoidance when \(L<n\).

Suppose \(L>n\).  Choose \(n\) leaves different from \(k\); this is possible
even when \(k\) is a leaf because then \(L-1\ge n\).  Use the chosen leaves
as singleton leaf blocks.  Removing them leaves a connected subtree
containing \(k\), which we take as the center block.  Contracting that block
to \(k\) gives a top-centered display of \(S_n\).

Now let \(L=n\).  If \(k\) is not a leaf, use all \(n\) leaves as singleton
leaf blocks.  Their complement is a connected subtree containing \(k\), so
it can be contracted to the center.  Conversely, suppose \(k\) is a leaf and
a top-centered display exists.  Its center block contains \(k\), while each
of the \(n\) quotient leaf blocks contains a distinct original leaf.  Only
\(n-1\) original leaves remain outside the center block, a contradiction.
\end{proof}

\begin{theorem}\label{thm:Aformula}
For $n\ge3$ and $k\ge2$,
\[
A^{\mathrm{top}}_n(k)=\sum_{L=0}^{n-1}\tau(k,L)+\frac{n}{k}\tau(k,n).
\]
For $k\ge n$, equivalently,
\[
A^{\mathrm{top}}_n(k)=
\sum_{L=0}^{n-1}\binom{k}{L}(k-L)!\stir{k-2}{k-L}
+\binom{k-1}{n-1}(k-n)!\stir{k-2}{k-n}.
\]
\end{theorem}

\begin{proof}
By Lemma~\ref{lem:leafcriterion}, the avoiders are exactly the trees with
fewer than $n$ leaves, together with the trees with exactly $n$ leaves in
which the largest label $k$ is a leaf.  The first class contributes
\[
\sum_{L=0}^{n-1}\tau(k,L),
\]
and the second is determined by labeling.  Counting
pairs consisting of such a tree and one of its leaves shows that any fixed
label is a leaf in
\[
 \frac{n}{k}\tau(k,n)
\]
of them.  Taking that label to be $k$ proves the first formula.  The second
follows by substituting the Pr\"ufer-code formula from Lemma~\ref{lem:tau}
and simplifying
\[
 \frac{n}{k}\binom{k}{n}
 =\binom{k-1}{n-1}.
\]
\end{proof}

\section{Generating functions and sharp asymptotics}\label{sec:starGF}
The leaf formula also determines the analytic form and leading asymptotic
size of the avoidance sequence. For background on Stirling numbers of the
second kind and their interpretation in terms of set partitions, see
\cite[Section~1.9]{Stanley1} and
\cite[Section~6.1]{GrahamKnuthPatashnik}. We include a proof of the
fixed-defect specialization needed below to identify its leading term and
fix our polynomial conventions.

\begin{definition}\label{def:star-egfs}
For $L\ge0$ and $n\ge3$, define
\[
T_L(x)=\sum_{k\ge2}\tau(k,L)\frac{x^k}{k!},
\qquad
A^{\mathrm{top}}_n(x)=\sum_{k\ge2}A^{\mathrm{top}}_n(k)\frac{x^k}{k!}.
\]
These are the exponential generating functions for fixed-leaf trees and top-centered star avoiders, respectively.
\end{definition}
\begin{theorem}\label{thm:EGF}
For $n\ge3$,
\[
A^{\mathrm{top}}_n(x)=\sum_{L=0}^{n-1}T_L(x)+n\int_0^x\frac{T_n(t)}{t}\,dt.
\]
\end{theorem}

\begin{proof}
The exact formula in Theorem~\ref{thm:Aformula} gives
\[
A^{\mathrm{top}}_n(k)=\sum_{L=0}^{n-1}\tau(k,L)+\frac{n}{k}\tau(k,n).
\]
Multiplying by $x^k/k!$ and summing over $k\ge2$, the first summand gives $\sum_{L=0}^{n-1}T_L(x)$.  For the second summand,
\[
\sum_{k\ge2}\frac{n}{k}\tau(k,n)\frac{x^k}{k!}
=n\sum_{k\ge2}\tau(k,n)\frac{x^k}{k\,k!}.
\]
On the other hand,
\[
\int_0^x\frac{T_n(t)}{t}\,dt
=\int_0^x\sum_{k\ge2}\tau(k,n)\frac{t^{k-1}}{k!}\,dt
=\sum_{k\ge2}\tau(k,n)\frac{x^k}{k\,k!}.
\]
The integral is a formal power series integral, and it is well-defined because $T_n(t)$ has zero constant term.  Combining the two contributions proves the formula.
\end{proof}

We next show that the normalized sequence $A_n^{\mathrm{top}}(k)/k!$ is a polynomial in $k$ up to a single $1/k$ correction.  The corresponding exponential generating function is therefore rational apart from one logarithmic term.

\begin{definition}\label{def:near-diagonal-stirling}
For $r\ge0$, the \emph{near-diagonal Stirling polynomial} $P_r(N)$ is defined by
\[
P_r(N)=\stir{N}{N-r}.
\]
For fixed $r$, this quantity is a polynomial in $N$; the polynomiality is proved in Theorem~\ref{thm:polylog}.
\end{definition}

\begin{theorem}\label{thm:polylog}
Fix $n\ge3$.  Then, for all $k\ge n+1$,
\[
\frac{A_n^{\mathrm{top}}(k)}{k!}=
\sum_{L=2}^{n-1}\frac{1}{L!}P_{L-2}(k-2)
+\frac{1}{(n-1)!k}P_{n-2}(k-2).
\]
Consequently there are a polynomial $Q_n(k)\in\mathbb Q[k]$ of degree
$2n-5$ and a rational number $\gamma_n$ such that, for $k\ge n+1$,
\[
A_n^{\mathrm{top}}(k)=k!\left(Q_n(k)+\frac{\gamma_n}{k}\right).
\]
Moreover
\[
[k^{2n-5}]Q_n(k)=\frac{1}{2^{n-2}(n-2)!(n-1)!}.
\]
Equivalently, the exponential generating function has the form
\[
A_n^{\mathrm{top}}(x)=R_n(x)-\gamma_n\log(1-x),
\]
where $R_n(x)\in\mathbb Q(x)$ has possible pole only at $x=1$, of order at most $2n-4$.
\end{theorem}

\begin{proof}
We first recall why the near-diagonal Stirling number is polynomial.  A partition of an $N$-element set into $N-r$ blocks is obtained by replacing some singleton blocks with non-singleton blocks of total deficit $r$.  If $\lambda\vdash r$ and $m_j(\lambda)$ is the number of parts of $\lambda$ equal to $j$, then the non-singleton blocks have sizes $j+1$, and
\[
P_r(N)=
\sum_{\lambda\vdash r}
\frac{(N)_{r+\ell(\lambda)}}
{\prod_{j\ge1}((j+1)!)^{m_j(\lambda)}m_j(\lambda)!}.
\]
Thus $P_r(N)$ is a polynomial of degree $2r$, with leading coefficient $1/(2^r r!)$; the leading term comes from the $r$ blocks of size $2$.

By Theorem~\ref{thm:Aformula},
\[
A_n^{\mathrm{top}}(k)=\sum_{L=0}^{n-1}\tau(k,L)+\frac{n}{k}\tau(k,n).
\]
For $L\ge2$, Lemma~\ref{lem:tau} gives
\[
\tau(k,L)=\binom{k}{L}(k-L)!\stir{k-2}{k-L}
=\frac{k!}{L!}P_{L-2}(k-2).
\]
Since $\frac{n}{k}\cdot\frac{1}{n!}=\frac{1}{(n-1)!k}$, this proves the displayed normalized formula.

Set
\[
\gamma_n=\frac{P_{n-2}(-2)}{(n-1)!}.
\]
Then
\[
\frac{P_{n-2}(k-2)}{k}
=
\frac{P_{n-2}(k-2)-P_{n-2}(-2)}{k}
+\frac{P_{n-2}(-2)}{k}.
\]
The numerator in the first fraction vanishes at $k=0$, so it is divisible by $k$ in $\mathbb Q[k]$.  Hence, for $k\ge n+1$, $A_n^{\mathrm{top}}(k)/k!$ has the form $Q_n(k)+\gamma_n/k$ for a polynomial $Q_n(k)$.  The highest-degree term comes from $P_{n-2}(k-2)/((n-1)!k)$.  Since $P_{n-2}$ has degree $2n-4$ and leading coefficient $1/(2^{n-2}(n-2)!)$, the leading coefficient of $Q_n$ is
\[
\frac{1}{2^{n-2}(n-2)!(n-1)!}.
\]

Apart from the finitely many coefficients with $2\le k<n+1$, the
exponential generating function is
\[
\sum_{k\ge2}Q_n(k)x^k+
\gamma_n\sum_{k\ge2}\frac{x^k}{k}.
\]
The first series is rational with possible pole only at $x=1$, of order at most $\deg Q_n+1=2n-4$, while the second differs from $-\gamma_n\log(1-x)$ by a polynomial.  The finitely many initial corrections are absorbed into the rational function.  This proves the stated rational-logarithmic form.
\end{proof}

\begin{lemma}\label{lem:fixedleafasymp}
For each fixed $L\ge2$,
\[
\tau(k,L)=\frac{k!}{L!\,2^{L-2}(L-2)!}k^{2L-4}+O_L(k!k^{2L-5}).
\]
\end{lemma}

\begin{proof}
Set $N=k-2$ and $r=L-2$.  We first show that
\[
\stir{N}{N-r}=\frac{N^{2r}}{2^r r!}+O_r(N^{2r-1}).
\]
A partition of $N$ elements into $N-r$ blocks has total deficit $r$, where a block of size $j$ contributes deficit $j-1$.  The leading contribution comes from the profile consisting of exactly $r$ blocks of size $2$ and all remaining blocks singletons.  The number of such partitions is
\[
\frac{1}{r!}\binom{N}{2}\binom{N-2}{2}\cdots\binom{N-2r+2}{2}
=\frac{N^{2r}}{2^r r!}+O_r(N^{2r-1}).
\]
Every other block-size profile with total deficit $r$ uses at most $2r-1$ elements in non-singleton blocks, and therefore contributes only $O_r(N^{2r-1})$ partitions.  This proves the displayed estimate for the near-diagonal Stirling number.  Since
\[
\tau(k,L)=\binom{k}{L}(k-L)!\stir{k-2}{k-L}=\frac{k!}{L!}\stir{k-2}{k-L},
\]
and $k-L=(k-2)-(L-2)=N-r$, substitution gives the result.
\end{proof}

\begin{theorem}\label{thm:sharpasymp}
For each fixed $n\ge3$,
\[
A^{\mathrm{top}}_n(k)=
\frac{k!}{2^{n-2}(n-2)!(n-1)!}k^{2n-5}
+O_n(k!k^{2n-6}).
\]
Consequently,
\[
\frac{A^{\mathrm{top}}_n(k)}{k^{k-2}}
=\frac{\sqrt{2\pi}}{2^{n-2}(n-2)!(n-1)!}
 e^{-k}k^{2n-5/2}\left(1+O_n(k^{-1})\right).
\]
\end{theorem}

\begin{proof}
By Theorem~\ref{thm:Aformula},
\[
A^{\mathrm{top}}_n(k)=\sum_{L=0}^{n-1}\tau(k,L)+\frac{n}{k}\tau(k,n).
\]
The terms with $L<2$ are zero for $k\ge2$.  By Lemma~\ref{lem:fixedleafasymp}, the largest term in the sum over $L\le n-1$ is $O_n(k!k^{2n-6})$.  The $n$-leaf term gives
\[
\frac{n}{k}\tau(k,n)
=\frac{n}{k}\left(
\frac{k!}{n!\,2^{n-2}(n-2)!}k^{2n-4}
+O_n(k!k^{2n-5})\right).
\]
Thus
\[
\frac{n}{k}\tau(k,n)=
\frac{k!}{2^{n-2}(n-2)!(n-1)!}k^{2n-5}
+O_n(k!k^{2n-6}),
\]
which proves the first assertion.  Dividing by $k^{k-2}$ and applying Stirling's formula
\[
k!=\sqrt{2\pi k}\,(k/e)^k\left(1+O(k^{-1})\right)
\]
gives the equivalent ratio estimate.
\end{proof}

\section{A threshold for growing top-centered stars}\label{sec:starthreshold}

The fixed-$n$ asymptotic concerns the restricted problem in which the
absolute largest host label $k$ survives as the center of the star. When
$n=n(k)$ grows, the question is whether a random host has enough leaves
to produce such a top-centered display. The Pr\"ufer code shows that the
transition occurs at $k/e$.

\begin{theorem}\label{thm:threshold}
Let $U_k$ be a uniformly random labeled tree on $[k]$.  Let $n=n(k)$ satisfy
$3\le n(k)<k$ for all sufficiently large $k$.  Then
\[
\mathbb P\bigl(U_k\text{ contains a top-centered display of }S_{n(k)}\bigr)
\to
\begin{cases}
1,&\text{if }\displaystyle \limsup_{k\to\infty}\frac{n(k)}{k}<\frac1e,\\[6pt]
0,&\text{if }\displaystyle \liminf_{k\to\infty}\frac{n(k)}{k}>\frac1e.
\end{cases}
\]
Thus the top-centered star-containment threshold for a random labeled tree on $[k]$ occurs at
\[
n\sim \frac{k}{e}.
\]
\end{theorem}

\begin{proof}
Let $L_k$ be the number of leaves of $U_k$.  By the leaf criterion, $U_k$ contains a top-centered display of $S_n$ whenever $L_k>n$, and avoids such a display whenever $L_k<n$.  In the boundary case $L_k=n$, containment depends on whether the largest label $k$ is a leaf.  Therefore, away from this boundary, the containment problem is controlled by $L_k$.

By the Pr\"ufer-code correspondence~\cite{Prufer}, $U_k$ is represented by a uniformly random word of length $k-2$ on the alphabet $[k]$.  A label is a leaf of $U_k$ if and only if it does not occur in this word.  Hence
\[
L_k=\sum_{i=1}^k I_i,
\]
where $I_i$ is the indicator that label $i$ is absent from the Pr\"ufer word.  We have
\[
p_k:=\mathbb P(I_i=1)=\left(1-\frac1k\right)^{k-2}
\]
and, for $i\ne j$,
\[
q_k:=\mathbb P(I_i=I_j=1)=\left(1-\frac2k\right)^{k-2}.
\]
Therefore
\[
\mathbb E[L_k]=kp_k.
\]
Since
\[
p_k=\left(1-\frac1k\right)^{k-2}=e^{-1}+O(k^{-1}),
\]
we get
\[
\mathbb E[L_k]=\frac{k}{e}+O(1).
\]
Next,
\[
\operatorname{Var}(L_k)
=k p_k(1-p_k)+k(k-1)(q_k-p_k^2).
\]
Using the expansions
\[
p_k=e^{-1}\left(1+\frac{3}{2k}+O(k^{-2})\right)
\]
and
\[
q_k=e^{-2}\left(1+\frac{2}{k}+O(k^{-2})\right),
\]
we obtain
\[
q_k-p_k^2=-\frac{e^{-2}}{k}+O(k^{-2}).
\]
Substitution gives
\[
\operatorname{Var}(L_k)=\left(e^{-1}-2e^{-2}\right)k+O(1),
\]
in particular $\operatorname{Var}(L_k)=O(k)$.  By Chebyshev's inequality,
\[
\frac{L_k}{k}\to \frac1e
\]
in probability.

Suppose first that
\[
\limsup_{k\to\infty}\frac{n(k)}{k}<\frac1e.
\]
Choose $\varepsilon>0$ such that, for all sufficiently large $k$,
\[
n(k)\le \left(\frac1e-\varepsilon\right)k.
\]
Since $L_k/k\to1/e$ in probability,
\[
\mathbb P(L_k>n(k))\to1.
\]
Whenever $L_k>n(k)$, Lemma~\ref{lem:leafcriterion} gives a top-centered display of $S_{n(k)}$.  Thus
\[
\mathbb P\bigl(U_k\text{ contains a top-centered display of }S_{n(k)}\bigr)\to1.
\]

Now suppose that
\[
\liminf_{k\to\infty}\frac{n(k)}{k}>\frac1e.
\]
Choose $\varepsilon>0$ such that, for all sufficiently large $k$,
\[
n(k)\ge \left(\frac1e+\varepsilon\right)k.
\]
Again using convergence in probability of $L_k/k$ to $1/e$, we have
\[
\mathbb P(L_k\ge n(k))\to0.
\]
Top-centered containment of $S_{n(k)}$ requires at least $n(k)$ leaves, so the containment probability tends to $0$.
\end{proof}

\section{Discussion and open problems}

The bounded-core theorem explains why display multiplicity is a finite
problem at each fixed collision order, but it does not yet recover the
support sequence.  Inclusion--exclusion mixes all collision orders, and the
survivor placements on a display still depend on the label set.
The main unresolved question is whether the resulting data determine the
target up to reversal.

The asymptotic theorem also shows where such information cannot be found.
For every fixed target, the proportion of containing trees tends to one
exponentially quickly.  Consequently, Conjecture~\ref{conj:RW} cannot be
settled by comparing limiting containment densities.  One must instead
distinguish the exponentially smaller avoidance sequences or extract
target-specific information from display collisions.

\begin{question}[Moment rigidity]
Does Wilf equivalence force equality of the second collision moments?
\[
\mu_{T_1}(m)=\mu_{T_2}(m)\text{ for every }m
\quad\Longrightarrow\quad
C_2(T_1;m)=C_2(T_2;m)\text{ for every }m?
\]
\end{question}

\begin{question}[Recovering pair cores]
Can the multiplicities of the survivor-decorated 2-cores be recovered
from finitely many values of $\mu_T(m)$?
\end{question}

\begin{remark}[Sharpness of the core bound]
For every $k\ge1$, the bound $k(n-1)+1$ is attained for every target
$T\in\T_n$ with $n\ge2$.
Replace each edge of $T$ by a path of $k$ edges. Keep the original vertices,
labeled $1,\ldots,n$, as the common survivors, and label the new vertices
with the remaining labels. Orient each replacement path. In display $a$,
leave its $a$th edge uncontracted and contract the segments on either side
toward their respective original endpoints. Each quotient is $T$, and
every host edge is crossing for exactly one display. Thus the host is
already reduced and has $k(n-1)+1$ vertices. For $n=1$, the reduced core
has one vertex. More generally, equality in the bound holds exactly when
the crossing-edge sets of the $k$ display maps are pairwise disjoint.
\end{remark}

For stars, the top-centered restriction turns containment into a leaf-count
condition and makes exact enumeration possible.  The corresponding ordinary
problem remains open: if any label may survive as the center, several
possible center blocks compete in the same host. For fixed $n\ge3$,
ordinary avoidance already has the same exponential rate as top-centered
avoidance. Indeed, let $A_n(k)$ denote the number of ordinary $S_n$-avoiders.
Every labeled path avoids $S_n$, and every ordinary avoider is also a
top-centered avoider. Hence
\[
\frac{k!}{2}\le A_n(k)\le A_n^{\mathrm{top}}(k)
=\Theta_n\!\left(k!\,k^{2n-5}\right).
\]
The remaining question is the precise asymptotic behavior of $A_n(k)$,
including its polynomial factor and leading constant.

\appendix

\section{Computational verification}\label{app:computations}

\subsection{The two-display example}\label{app:two-display-enumeration}
For the two-display example in Section~\ref{sec:decorated-cores}, write
\[
    A=(1,1),\qquad B=(1,2),\qquad
    C=(2,1),\qquad D=(2,2)
\]
for the four survivor marks. In the table below, a subscript on a core
vertex records its full label. Thus $x_{ij}$ means
\[
    \Phi(x)=(i,j).
\]
The notation $BC_y$ means that $\{B,C\}$ is one class of $P$ and that
$\beta(\{B,C\})=y$. Singleton classes are written in the same way. Only
the core types having a nonzero lift count at $m=3$ are shown.

\begin{center}
\small
\begin{tabular}{cclccc}
\hline
Type & Underlying labeled core & Survivor classes and locations
 & $L_3^{\mathrm{rig}}$ & $|\operatorname{Aut}|$ & Direct count\\
\hline
1  & $x_{12}-y_{21}$
   & $A_x,BC_y,D_x$ & 2 & 1 & 2\\
2  & $y_{12}-x_{11}-z_{21}$
   & $A_x,BC_z,D_y$ & 1 & 1 & 1\\
3  & $y_{11}-x_{21}-z_{22}$
   & $A_y,BC_x,D_z$ & 1 & 1 & 1\\
4  & $y_{12}-x_{22}-z_{21}$
   & $A_y,BC_z,D_x$ & 1 & 1 & 1\\
5  & $x_{11}-y_{22}$
   & $A_x,BD_y,C_x$ & 4 & 1 & 4\\
6  & $y_{11}-x_{12}-z_{22}$
   & $A_x,BD_z,C_y$ & 2 & 1 & 2\\
7  & $y_{11}-x_{21}-z_{22}$
   & $A_y,BD_z,C_x$ & 2 & 1 & 2\\
8  & $y_{12}-x_{22}-z_{21}$
   & $A_y,BD_x,C_z$ & 2 & 1 & 2\\
9  & $x_{11}-y_{22}$
   & $AC_x,B_y,D_y$ & 4 & 1 & 4\\
10 & $y_{12}-x_{11}-z_{21}$
   & $AC_x,B_z,D_y$ & 2 & 1 & 2\\
11 & $y_{11}-x_{12}-z_{22}$
   & $AC_y,B_z,D_x$ & 2 & 1 & 2\\
12 & $y_{11}-x_{21}-z_{22}$
   & $AC_y,B_x,D_z$ & 2 & 1 & 2\\
13 & $y_{11}-x_{12}-z_{22}$
   & $AC_y,BD_z$ & 3 & 1 & 3\\
14 & $y_{11}-x_{21}-z_{22}$
   & $AC_y,BD_z$ & 3 & 1 & 3\\
15 & $x_{12}-y_{21}$
   & $AD_x,B_y,C_y$ & 2 & 1 & 2\\
16 & $y_{12}-x_{11}-z_{21}$
   & $AD_y,B_z,C_x$ & 1 & 1 & 1\\
17 & $y_{11}-x_{12}-z_{22}$
   & $AD_x,B_z,C_y$ & 1 & 1 & 1\\
18 & $y_{12}-x_{22}-z_{21}$
   & $AD_y,B_x,C_z$ & 1 & 1 & 1\\
\hline
   & & Total & 36 & & 36\\
\hline
\end{tabular}
\end{center}

All automorphism groups in this example are trivial. The full labels and
the survivor locations distinguish any vertices that could otherwise be
interchanged. The table therefore gives
\[
    \sum_{\substack{
       [\Omega]\in\mathcal C_2(T)\\
       L_3^{\mathrm{rig}}(\Omega)>0}}
    \frac{L_3^{\mathrm{rig}}(\Omega)}
         {|\operatorname{Aut}(\Omega)|}
    =36
    =\widehat C_2(T;3),
\]
in agreement with direct enumeration.

\subsubsection*{Enumeration method and output}
All trees on $[3]$ are generated, and all possible survivor pairs and display
maps are tested directly. Two core records are identified only when a
permutation of the core vertices preserves adjacency, both quotient maps,
every named survivor class, and the map $\beta$.

\begin{verbatim}
for each tree U on [3]
    D = empty list

    for each pair r1 < r2
        for each map pi from [3] to {1,2}
            require pi(r1) = 1 and pi(r2) = 2
            if both fibers of pi are connected
                add (pi,r1,r2) to D

    for each ordered pair (delta1,delta2) of distinct elements of D
        contract every edge on which pi1 and pi2 are both constant
        record the induced maps phi1 and phi2
        partition A,B,C,D by equality of their survivor labels
        record the core vertex containing each survivor class
        canonicalize over all permutations of the core vertices
        increase the direct count of the resulting core type

for each recorded core Omega
    Lrig[Omega] = 0

    for each surjection b from [3] to V(H)
        let Bx be the inverse images of b

        for each injection theta from P to [3]
            require theta(C) to lie in B_beta(C)
            require theta(C_a,1) < theta(C_a,2) for a = 1,2

            add
                product over x of
                |Bx|^(|Bx|-2+d_H(x))
            to Lrig[Omega]
\end{verbatim}

The resulting output is
\begin{verbatim}
number of labeled host trees          3
display counts                        4, 4, 4
ordered distinct display pairs        12, 12, 12
contributing decorated-core types     18
sum of direct counts                  36
sum of Lrig/|Aut|                     36
\end{verbatim}

\subsection{Second-moment comparison for labeled stars}
\label{app:second-moment-enumeration}
For $r\geq0$, put
\[
N_r(T)=\#\{U\in\mathcal T_5:c_T(U)=r\}.
\]
A display of a four-vertex target in a five-vertex host contracts exactly
one oriented edge. We generated the $5^3=125$ labeled trees on $[5]$ from
their Pr\"ufer words and tested both orientations of each of their four edges.
For the targets $T_1$ and $T_2$ defined in Section~\ref{sec:decorated-cores}, the resulting display
distributions are
\[
\begin{array}{c|rrrrrr}
r&0&1&2&3&4&5\\ \hline
N_r(T_1)&93&15&1&12&3&1\\
N_r(T_2)&89&19&4&9&4&0
\end{array}
\]
and hence
\[
C_1(T_1;5)=C_1(T_2;5)=70,
\]
whereas
\[
C_2(T_1;5)
 =\sum_{r\geq0}\binom r2N_r(T_1)=65
\]
and
\[
C_2(T_2;5)
 =\sum_{r\geq0}\binom r2N_r(T_2)=55.
\]
The numbers of hosts with at least one display are
\[
\mu_{T_1}(5)=\sum_{r\geq1}N_r(T_1)=32,
\qquad
\mu_{T_2}(5)=\sum_{r\geq1}N_r(T_2)=36.
\]
The enumeration can be summarized as follows.
\begin{verbatim}
for each Prufer word w in [5]^3
    construct the labeled tree U encoded by w

    for each oriented edge a -> b of U
        contract a toward b and increasingly standardize the result
        if the result is T_i
            increase c_{T_i}(U)

    increase N_{c_{T_i}(U)}(T_i)
\end{verbatim}

\section*{Acknowledgements}
The author thanks Marigold Strupp for valuable discussions and contributions to star
enumeration, and Amrit Kandasamy for valuable discussions and contributions for asymptotics of stars. The author thanks Lalith Durbhakula for helpful discussions
and preliminary computational work on labeled-tree avoidance, contraction containment, and
small Wilf-equivalence data. The author thanks Simon Rubinstein-Salzedo for his mentorship
and guidance throughout this research project and all of the members of the Euler Circle
Combinatorics Research Group who offered guidance or help throughout this research.
\section*{Declaration of generative AI}

During the preparation of this work, the author used ChatGPT 5.6 Pro (OpenAI) to assist with manuscript editing and literature searches. After using this tool, the author reviewed and edited the content as needed and takes full responsibility for the content of the published article.

\end{document}